\documentclass[reqno]{amsart}
\usepackage{amssymb}
\usepackage{amsmath}
\usepackage{verbatim}
\usepackage{enumerate}
\usepackage[usenames, dvipsnames]{color}

\usepackage{hyperref}

\newtheorem{theorem}{Theorem}[section]
\newtheorem{lemma}[theorem]{Lemma}
\newtheorem{corollary}[theorem]{Corollary}
\newtheorem{remark}[theorem]{Remark}
\newtheorem{proposition}[theorem]{Proposition}

\theoremstyle{definition}
\newtheorem{definition}[theorem]{Definition}

\newtheorem{assumption}[theorem]{Assumption}
\newtheorem{condition}[subsubsection]{Condition}

\theoremstyle{remark}

\newcommand{ \mr }{ \mathbb{R} }

\newcommand{ \mch }{ \mathcal{H} }

\numberwithin{equation}{section}

\begin{document}

\title[Parabolic equations with unbounded lower-order coefficients]{Parabolic equations with unbounded lower-order coefficients in Sobolev spaces with mixed norms}

\author{Doyoon Kim}
\address{Department of Mathematics, Korea University, 145 Anam-ro, Seongbuk-gu,
Seoul, 02841, Republic of Korea}
\email{doyoon\_kim@korea.ac.kr}
\thanks{D. Kim and K. Woo were supported by the National Research Foundation of Korea (NRF) grant funded by the Korea government (MSIT) (2019R1A2C1084683).}

\author{Seungjin Ryu}
\address{Department of Mathematics, University of Seoul, 163 Seoulsiripdaero, Dongdaemun-gu, 
Seoul, 02504, Republic of Korea}
\email{seungjinryu@uos.ac.kr}
\thanks{S. Ryu was supported by NRF-2017R1C1B1010966 and NRF-2020R1C1C1A01014310.}

\author{Kwan Woo}
\address{Department of Mathematics, Korea University, 145 Anam-ro, Seongbuk-gu,
Seoul, 02841, Republic of Korea}
\email{wkwan@korea.ac.kr}



\begin{abstract}
We prove the $L_{p,q}$-solvability of parabolic equations in divergence form with full lower order terms. The coefficients and non-homogeneous terms belong to mixed Lebesgue spaces with the lowest integrability conditions.
In particular, the coefficients for the lower-order terms are not necessarily bounded.
We study both the Dirichlet and conormal derivative boundary value problems on irregular domains.
We also prove embedding results for parabolic Sobolev spaces, the proof of which is of independent interest.
\end{abstract}

\maketitle

\section{Introduction}

There are many well-established results for $L_p$-theory on linear parabolic equations of the form
\begin{equation}
							\label{eq0407_01}
u_t - D_i \left( a^{ij}(t,x) D_j u +a^i(t,x) u \right)  + b^i(t,x) D_i u  +c(t,x) u = D_ig_i(t,x) + f(t,x).
\end{equation}
That is, for given $g_i, f \in L_p$, or $g_i, f \in L_{p,q}$, many papers  investigate whether there exists a unique solution $u$ of \eqref{eq0407_01} in an appropriate parabolic Sobolev space under some regularity assumptions on the coefficients $a^{ij}$, $a^i$, $b^i$ and $c$.
See, for instance, \cite{MR3073000, MR2680179, MR2835999, MR2764911}, where their main objective is to obtain the lowest possible regularity assumptions on $a^{ij}$ (and on the boundary of the domain) for the unique solvability of equations or systems in Sobolev spaces, while the lower-order coefficients $a^i$, $b^i$, and $c$ are assumed to be either zero or bounded.
Indeed, if the lower-order coefficients are bounded, it is not difficult to derive the desired results whenever one obtains the unique solvability of equations or systems without lower-order terms.

One may ask if the same $L_p$-theory can be established for equations as in \eqref{eq0407_01} with unbounded lower-order coefficients.
To the best of the authors' knowledge, there seem few papers investigating $L_p$-theory for parabolic equations as in \eqref{eq0407_01} having possibly unbounded lower-order coefficients.
One possible reference about $L_2$-theory, not $L_p$-theory, would be the context of \cite[Chapter\,III]{MR0241822}, where the authors discussed the existence and uniqueness to the generalized solution of \eqref{eq0407_01} with a proper boundary condition, under the assumptions that $a^{ij}$ is uniformly elliptic and bounded and that
$$
a^i, \ b^i, \ |c|^{1/2} \in L_{q,r},     \qquad    g_i \in L_{2,2},        \qquad     f \in L_{q_1, r_1}
$$
for $i=1,\ldots, d$. 
Here, if $d\geq3$, the indices satisfy the conditions
$$  
\frac{d}{q}+\frac{2}{r} \leq 1, \qquad q \in [d, \infty], \quad r \in  [2,\infty],
$$
and
$$
\frac{d}{q_1} + \frac{2}{r_1}  \leq 2 +\frac{d}{2}, \qquad q_1 \in \left[ \frac{2d}{d+2}, 2 \right], \qquad r_1 \in [1,2].
$$
An additional smallness condition is imposed on $a^i, b^i, c$ when $r = \infty$.
They also considered the cases $d=1,2$. Note that the coefficients $a^i$, $b^i$, and $c$ are not necessarily bounded.
On the other hand, for qualitative properties of solutions of parabolic (and elliptic) equations with unbounded lower-order coefficients, see \cite{MR2760150}.
In this paper, we establish an $L_{p,q}$ version of the results developed in \cite{MR0241822}.
This in particular means that the results in \cite{MR0241822} can be covered as a special case of our main results when $p = q = 2$ (see Remark \ref{rmk206}.).
We find in this paper the optimal spatial and time integrability conditions for the lower-order coefficients as functions in Lebesgue spaces with mixed norms to obtain the unique solvability of the equation \eqref{eq0407_01} in parabolic Sobolev spaces with mixed norms.
As a necessary step in doing so, we also find optimal spatial and time summability conditions on $f$ in \eqref{eq0407_01}.
We deal with both the Dirichlet and conormal derivative boundary conditions.

As is shown in \cite{MR3488249}, when establishing $L_p$-theory, it is necessary to impose some regularity assumptions on the leading coefficients $a^{ij}$ even if the lower-order coefficients are zero or bounded.
In fact, there have been many research activities in finding regularity assumptions on $a^{ij}$ (for instance, much less irregular than uniformly continuous) so that one can establish $L_p$-theory for equations with such leading coefficients.
One well-known class of possibly discontinuous coefficients is that of vanishing mean oscillations (VMO) \cite{MR1191890, MR1088476, MR1239929, MR2304157}.
It is also possible to have more irregular ones than VMO coefficients.
Some of such irregular coefficients are those merely measurable in one spatial variables, that is, the coefficients are allowed to have no regularity assumptions as functions of one spatial variable.
See \cite{MR2835999, MR3073000} and references therein.
Moreover, in the $L_p$-theory for  parabolic equations/systems, one can further relax regularity assumptions so that all the elements of the matrix $a^{ij}(t,x)$ except, for example, $a^{11}(t,x)$ can be merely measurable in the time variable and one spatial variable.
The coefficient $a^{11}(t,x)$ is only measurable in $t$ or in one spatial variable.
See \cite{MR2764911} for more details.
Concerning the regularity of the domain, we do need some conditions unless we confine ourselves to $L_2$-theory.
See \cite{MR1331981}.
Recent results show that one can deal with the equation \eqref{eq0407_01} in $\Omega$ if $\Omega$ is a Reifenberg flat domain having  sufficient flatness.
See \cite{MR2680179, MR2835999} and references therein.
In this paper, because the main objective of our investigation is to find appropriate summability conditions for lower-order coefficients so that one can embrace unbounded lower-order coefficients in $L_p$-theory for parabolic equations, we do not specify regularity conditions for $a^{ij}$ as well as for the boundary of the domain.
Instead, we take, as an assumption, the solvability of the equation \eqref{eq0407_01} {\em without} lower-order terms in appropriate Sobolev spaces with mixed norms.
See Assumption \ref{aomega} for the Dirichlet boundary condition case and Assumption \ref{comega} for the conormal derivative boundary condition case.
Note that in these assumptions the function $f$ on the right-hand side of \eqref{eq0407_01} has the same summability as those of $g_i$ and the solutions.
These assumptions indicate that the leading coefficients $a^{ij}$ in this paper can be those available in the literature as long as $L_p$-theory for equations with $a^{ij}$ and with zero or bounded lower-order coefficients has been successfully established.
In this way, without tediously stating legitimate regularity assumptions on $a^{ij}$, we can cover all the irregular leading coefficients  in the literature at once.
Similarly, instead of imposing a particular regularity assumption on the domain, we state an assumption on the domain.
See Assumption \ref{dom} and the comment above the assumption.
To justify our assumption on the domain, we prove that Reifenberg flat domains, well-known irregular domains for $L_p$-theory, satisfy Assumption \ref{dom}.
We also present the embedding results for the whole Euclidean space and a half space.
In fact, the embedding results proved in Appendix (Section \ref{appendix}) are another main results of this paper because the time regularity conditions required for the embedding are considerably more general than those in the literature.

When Sobolev spaces with {\em mixed norms} are considered as solutions spaces for parabolic equations/systems, one can find the unique solvability results in \cite{MR2764911, MR2352490}, where, as noted earlier, all the lower-order coefficients are assumed to be bounded.
On the other hand, in the elliptic case, there are relatively many results about the unique solvability in the Sobolev space $W_p^1$. See, for instance, \cite{MR3900848, MR3623550, MR3328143}.  
In particular, assuming that $c=0$ and that $a^i, b^i \in L_r(\Omega)$, where $2<r<\infty$ if $d=2$ and $d \leq r < \infty$ if  $d \geq 3$,
the authors of \cite{MR3900848, MR3623550} proved $W^1_p$-estimates when $\Omega$ has a small Lipschitz constant and $a^{ij}(x)$ have small BMO semi-norms in the $x$ variables.
Their approach is a functional-analytic argument and is influenced by the method presented in \cite{MR1908676, MR2476418}.
For non-divergence type elliptic equations with unbounded lower-order coefficients, see a recent paper \cite{arXiv:2004.01778} by Krylov.
Our approach is more elementary.
It is based on an anisotropic embedding inequality (see Assumption \ref{dom} ($\Omega$) as well as Theorems \ref{thm602} and \ref{thm702}) and the duality argument (see Lemma \ref{190405@lem1}).
When applying the duality argument, we use Assumption \ref{aomega} ($p,q$) to guarantee 
the solvability of equations without lower-order terms for the full range $1 < p,q <\infty$.
 
The main objective in the duality argument is to resolve the lower integrability  issue on $f$ in \eqref{eq0407_01}, which is not considered in the previous results \cite{MR2764911, MR3900848, MR3623550, MR3328143, MR2352490}.
For elliptic equations, as noticed in \cite{MR1405255}, one may use the harmonic extension method to lower the required summability on $f$.
For instance, in the elliptic case with the Dirichlet boundary condition, $f$ can be replaced by $D_i(D_iF)$, where $F$ is a solution to $\Delta F = f \chi_\Omega$ in a ball $B \supset \Omega$ with the estimate $\| D_i F \|_{L_p} \le C \| f \|_{ L_{p_1} }$, $p_1 = dp/(d+p)$.
However, this idea is not working for parabolic equations unless the boundary of the domain is sufficiently regular or the summability of $\|f(t,\cdot)\|_{L_{p_1}(\Omega)}$ in time is sufficiently high.
Once we have mixed-norm estimates for equations without lower-order terms, but with $f$ having low integrability, by moving the lower-order terms to the right-hand side of the equations we derive desired estimates for equations as in \eqref{eq0407_01} with full lower order terms.
After that, we show that such solutions have better regularity with respect to the time variable if $p,q \in [2,\infty)$ by showing that solutions are in fact in the parabolic space $V_{p,q}$. 
See the definitions of the spaces such as $V_{p,q}$ and $\mathcal{H}^1_{p,q}$ in Section \ref{main result}.
We remark that $V_{p,q}$ is not appropriate for the solvability with $p, q \in (1, 2)$, while
$\mathcal{H}^1_{p,q}$ is not wide enough to accommodate solutions of equations with unbounded lower-order coefficients with $p,q \in (2,\infty)$. 
See Remark \ref{counter2}.

The rest of the paper is organized as follows. 
In Section \ref{main result}, we state our assumptions and main results after introducing some function spaces.
In Section \ref{dirichlet} we present the proofs of our main results for the Dirichlet and the conormal cases after providing  the key lemma (Lemma \ref{190405@lem1}) for equations without lower-order terms.  
Section \ref{Vspace} discusses better time regularity of solutions for $p,q \in [2,\infty)$.
Finally, we present domains (with a detailed proof for Reifenberg flat domains) satisfying Assumption \ref{dom} in Section \ref{appendix}, where one can also find a brief comment on the elliptic case.

\section{ Assumptions and main results}
\label{main result}

We denote by $\mr^d$, where $d$ is a positive integer, a $d$-dimensional Euclidean space and a point in $\mr^d$ by $x = (x_1, \ldots, x_d) = (x_1, x')$. 
Similarly, $\mr^{d+1}:= \mr \times \mr^d = \{(t, x) : t \in \mr, \,x \in \mr^d\}$ and $\mr_+^d := \{(x_1, x') : x_1 > 0, \,x' \in \mr^{d-1}\}$. Throughout the paper, we assume $T$ to be a positive real number, unless specified otherwise.
We set $\Omega_T := (0,T) \times \Omega$, where $\Omega$ is a domain in $\mr^d$.
For $p, q \in [1,\infty)$, $L_{p,q}(\Omega_T)$ is the set of all measurable functions on $\Omega_T$ with a finite norm
$$
\|v\|_{L_{p,q}(\Omega_T)}=
\left(\int_0^{T}\left(\int_{\Omega}|v(t,x)|^p\,dx\right)^{q/p}\,dt\right)^{1/q}.
$$
For $k\in \{1,2\}$, we set
$$
W^{k,1}_{p,q}(\Omega_T)=\{v:v_t,\,  D^l v \in L_{p,q}(\Omega_T), \, | l |\le k \}
$$
equipped with a norm
$$
\|v\|_{W^{k,1}_{p,q}(\Omega_T)}=\|v_t\|_{L_{p,q}(\Omega_T)}+\sum_{| l |\le k}\|D^l v\|_{L_{p,q}(\Omega_T)}.
$$
We also set
$$
W^{1,0}_{p,q}(\Omega_T)  = \left\{ v :  v, Dv \in  L_{p,q}(\Omega_T)  \right\}
$$
with
$$
\|  v\|_{W^{1,0}_{p,q}(\Omega_T)} = \| v \|_{L_{p,q}(\Omega_T)} + \| Dv \|_{L_{p,q}(\Omega_T)} .
$$
By $v\in \mathbb{H}^{-1}_{p,q}(\Omega_T)$ we mean that there exist $g=(g_1,\ldots,g_d)\in L_{p,q}(\Omega_T)^d$ and $f\in L_{p,q}(\Omega_T)$ such that 
$$
v=D_i g_i +f   \qquad  \text{ in } \quad  \Omega_T
$$
in the distribution sense and the norm
$$
\|v\|_{\mathbb{H}^{-1}_{p,q}(\Omega_T)} =\inf\left\{ \|g\|_{L_{p,q}(\Omega_T)}+\|f\|_{L_{p,q}(\Omega_T)}: v=D_i g_i +f \right\}
$$
is finite.

We define $\mch^1_{p,q}(\Omega_T)$ and $V_{p,q}(\Omega_T)$ as the sets of all functions in $W^{1,0}_{p,q}(\Omega_T)$ having the following finite norms, respectively,
$$
\|v\|_{\mch^1_{p,q}(\Omega_T)}:=\|v_t\|_{\mathbb{H}^{-1}_{p,q}(\Omega_T)}+\|v\|_{W^{1,0}_{p,q}(\Omega_T)},
$$
$$
\|v\|_{V_{p,q}(\Omega_T)}:=\operatorname*{ess\,sup}_{0<t<T} \|v(t, \cdot)\|_{L_p(\Omega)}+\|v\|_{W^{1,0}_{p,q}(\Omega_T)}.
$$

We denote by $\mathring{W}^{1,0}_{p,q}(\Omega_T)$, $\mathring{\mch}^1_{p,q}(\Omega_T)$, and 
$\mathring{V}_{p,q}(\Omega_T)$ the closures of $C^\infty_0([0, T] \times \Omega)$ in $W^{1,0}_{p,q}(\Omega_T)$, $\mch^1_{p,q}(\Omega_T)$, and $V_{p,q}(\Omega_T)$, respectively, where $C^\infty_0( [0, T] \times \Omega)$ is the set of all infinitely differentiable functions on $[0, T] \times \Omega$ with compact support in $[0, T] \times \Omega$.

The equation we consider in this paper is
\begin{equation}
							\label{eq0402_01}
\mathcal{P} u = D_i g_i + f
\end{equation}
in $\Omega_T$,
where
$$
\mathcal{P}u = -u_t +D_i \left( a^{ij}(t,x) D_j u +a^i(t,x) u \right)  + b^i(t,x) D_i u  +c(t,x) u.
$$

\subsection{Assumption on domain}

We assume that the domain $\Omega$ enjoys the embedding inequality \eqref{imb}, which is indeed satisfied by domains such as Lipschitz domains and Reifenberg flat domains.
Instead of specifying the regularity of the boundary (that is, instead of assuming that the domain is, for instance, a Reifenberg flat domain), we require the domain to satisfy the embedding inequality, which can be derived under a sufficient regularity assumption on the boundary of the domain.
In this way, we can include various classes of domains in our results, provided that their boundaries are regular enough to insure the embedding inequality.
In Section \ref{appendix} we present examples of domains satisfying Assumption \ref{dom}.
In particular, we prove there that  a Reifenberg flat domain satisfies Assumption \ref{dom}.

\begin{assumption}[$\Omega$]
\label{dom}
Let  $\Omega$ be a domain in $\mr^d$ satisfying the following. 
For $u \in W_{p, q}^{1, 0}(\Omega_T)$ with $p, q \in [1,\infty]$, suppose
$$
u_t = D_ig_i + \sum_{k=1}^{m}f_k \quad \textrm{in} \quad \Omega_T
$$
in the distribution sense where $g = (g_1, \ldots, g_d) \in \left( L_{p, q}(\Omega_T) \right)^d$ and $f_k \in L_{p_k, q_k}(\Omega_T)$ with $p_k, q_k  \in (0,\infty]$, $k = 1, \ldots, m$.
Then we have
\begin{equation}
\label{imb}
\|u\|_{L_{p_0,q_0}(\Omega_T)} \leq C \left( \|u\|_{W_{p, q}^{1, 0}(\Omega_T)} + \|g\|_{L_{p, q}(\Omega_T)} + \sum_{k=1}^m \|f_k\|_{L_{p_k, q_k}(\Omega_T)} \right), 
\end{equation}
provided that $p_0 \in [p,\infty]$, $q_0 \in [q,\infty]$, $p_k \in (0,\infty]$, and $q_k \in (0,\infty]$ satisfy
either (i) or (ii) of the following conditions.
\begin{enumerate}
\item[(i)] If $q_0 = q$, then
\begin{equation}
							\label{eq0602_03}
\frac{d}{p} \leq 1 + \frac{d}{p_0}, \quad (p,p_0) \neq \left( d(\geq 2), \infty\right).
\end{equation}
In this case $p_k$ and $q_k$, $k = 1, \ldots, m$, are arbitrary.

\item[(ii)] If $q_0 > q$, then
\begin{equation}
							\label{eq0602_01}
\frac{d}{p} + \frac{2}{q} \leq 1 + \frac{d}{p_0} + \frac{2}{q_0},
\end{equation}
and
\begin{equation}
	\label{eq0602_02}
1 < q < q_0 < \infty \quad \textrm{if} \quad \frac{d}{p} + \frac{2}{q} = 1 + \frac{d}{p_0} + \frac{2}{q_0}.
\end{equation}
In this case, $p_k$ and $q_k$ are  real numbers such that
\begin{equation}
							\label{eq0603_01}
p_k \in [1,   p_0], \quad q_k \in [1,  q_0],
\end{equation}

\begin{equation}
							\label{eq0603_02}
\frac{d}{p_k} + \frac{2}{q_k} \leq 2 + \frac{d}{p_0} + \frac{2}{q_0}.
\end{equation}
\begin{equation}
	\label{eq0603_03}
1 = q_k < q_0 = \infty  \quad \textrm{or}  \quad 1 < q_k < q_0 < \infty \quad \textrm{if} \quad \frac{d}{p_k} + \frac{2}{q_k} = 2 + \frac{d}{p_0} + \frac{2}{q_0}.
\end{equation}
\end{enumerate}

The constant $C = C_{p, q, p_0, q_0, p_k, q_k}$ in \eqref{imb} may vary depending on the choice of eligible $(p,q,p_0,q_0, p_k, q_k)$, but independent of $u$.
For simplicity, we may write $C_k$ for $C_{p, q, p_0, q_0, p_k, q_k}$ or for the constant $C$ when $p, q, p_0, q_0, p_k, q_k$ are replaced by some related parameters.
\end{assumption}

\begin{remark}
							\label{rem0602_1}
Note that if $(p_k,q_k)$ is given by $(p,q)$ in Assumption \ref{dom}, then $(p,q)$ satisfies \eqref{eq0603_01} and \eqref{eq0603_02} (the borderline case \eqref{eq0603_03} is excluded by the condition \eqref{eq0602_01}), provided that $(p,q)$ and $(p_0,q_0)$ satisfy the conditions in Assumption \ref{dom}.
Thus, the estimate \eqref{imb} holds if $(p_k,q_k)$ is replaced by $(p,q)$.
\end{remark}

\begin{remark}
							\label{rem0430_1}
If $\Omega = \mr^d$, $p, q \in (1, \infty)$, and $(p_k, q_k) = (p, q)$, $k = 1, \ldots, m$, the embedding inequality in Assumption \ref{dom} is easily obtained from the embedding inequality for $W_{p,q}^{2,1}(\mr^d_T)$.
See, for instance, \cite{MR519341}.
However, even if $u \in \mathcal{H}_{p,q}^1(\Omega_T)$ vanishes on the lateral boundary of $\Omega_T$, one may not obtain the inequality \eqref{imb} by applying the embedding result for $\mr^d$ to the zero extension $\bar{u}$ of $u$, that is,
$$
\bar{u} = \left\{
\begin{aligned}
u \quad &\text{in} \quad (t,x) \in \Omega_T,
\\
0 \quad &\text{in} \quad (t,x) \in (0,T) \times \left(\mr^d \setminus \Omega\right),
\end{aligned}
\right.
$$
because $\bar{u}$ is not necessarily in $\mathcal{H}_{p,q}^1(\mr^d_T)$.
\end{remark}

\subsection{Assumptions on the non-homogeneous terms}
\label{condi0302_2}


For $p, q \in (1, \infty )$, we consider \eqref{eq0402_01} in $\Omega_T$ with $g_i \in L_{p,q}(\Omega_T)$, $i=1,\ldots,d$, and $f \in L_{p_1,q_1}(\Omega_T)$, where $p_1$ and $q_1$ are positive real numbers such that
\begin{equation}
	\label{40111}
\left(p_1, q_1, p\right) \ne \left(1, q, \frac{d}{d-1} \right) \quad \textrm{if} \quad d \ge 2, \quad p_1 \in [1, p], \quad q_1 \in [1, q],
\end{equation}
\begin{equation}
\label{401}
\frac{d}{p_1} + \frac{2}{q_1} \le 1+\frac{d}{p} + \frac{2}{q},
\end{equation}
and
\begin{equation}
	\label{4011}
q_1 >1 \quad \textrm{if} \quad \frac{d}{p_1} + \frac{2}{q_1} = 1+\frac{d}{p} + \frac{2}{q}.
\end{equation}

\subsection{Assumptions on the coefficients}
 Throughout the paper, we assume that the leading coefficients $a^{ij}$ of the operator $\mathcal{P}$ in \eqref{eq0402_01} satisfy the following. There exists a constant $\delta \in (0,1)$ such that
the coefficients $a^{ij}(t,x)$, $1 \le i,j \le d$, satisfy
\begin{equation}
							\label{eq0402_02}
 | a^{ij}(t,x) | \le \delta^{-1} \quad \textrm{ and } \quad  a^{ij}(t,x) \xi_i \xi_j  \ge \delta | \xi|^2
\end{equation}
for all $(t,x) \in \mr \times \mr^{d}$ and $\xi=(\xi_1, \ldots, \xi_d) \in \mr^d$.

\subsubsection{Algebraic conditions on the summability for lower-order coefficients}
\label{condi0302_1}

For $p,q\in (1,\infty)$ and $i = 1, \ldots , d$, we consider the lower-order coefficients $a^i$, $b^i$, and $c$ of the operator $\mathcal{P}$ such that
\begin{equation}
							\label{eq0402_03}
a^i \in L_{\ell_1, r_1}(\Omega_T),  \quad
b^i \in L_{\ell_2, r_2}(\Omega_T), \quad 
c \in L_{\ell_3, r_3}(\Omega_T),
\end{equation}
where the pairs $(\ell_k, r_k)$, $k\in \{1,2,3\}$, are determined by $d, p, q$ as follows.

Set $\ell_1 \in (d, \infty] \cap [p, \infty]$ and $r_1 \in [2,\infty) \cap [q, \infty)$ such that
$$
\frac{d}{\ell_1} + \frac{2}{r_1} \leq 1 \quad \text{and} \quad
r_1 > q \quad \text{if} \quad \frac{d}{\ell_1} + \frac{2}{r_1} = 1.
$$
Note that
\begin{equation*}
\frac{\ell_1 p}{\ell_1 - p} \in [p,\infty], \quad \frac{r_1q}{r_1-q} \in (q,\infty]
\end{equation*}
and the pair $\displaystyle\left(\frac{\ell_1p}{\ell_1 - p}, \frac{r_1q}{r_1 - q} \right)$ satisfies \eqref{eq0602_01} and \eqref{eq0602_02} in place of $(p_0, q_0)$. 

Set $\ell_2 \in (d, \infty] \cap [{p}/{(p-1)}, \infty]$ and $r_2 \in [2,\infty) \cap [{q}/{(q-1)}, \infty)$ such that
$$
\frac{d}{\ell_2} + \frac{2}{r_2} \leq 1 \quad \text{and} \quad r_2 > \frac{q}{q-1} \quad
\text{if} \quad \frac{d}{\ell_2} + \frac{2}{r_2} = 1.
$$
Note that the pair $\displaystyle\left(\frac{\ell_2p}{\ell_2 + p}, \frac{r_2q}{r_2 + q}\right)$ satisfies \eqref{40111}--\eqref{4011} in place of $(p_1, q_1)$.
In particular, we have $\displaystyle\frac{r_2q}{r_2 + q} \in [1,q)$ instead of $\displaystyle\frac{r_2q}{r_2 + q} \in [1,q]$.
Thus, the first condition in \eqref{40111} is trivially satisfied.

The pair $(\ell_3,r_3)$ satisfies either Condition \ref{cond0625_1} or Condition \ref{cond0625_2} below.

\begin{condition}
							\label{cond0625_1}
The pair $(\ell_3,r_3)$ with $\ell_3 \in [1,\infty]$ and $r_3 \in [1,\infty)$ satisfies
$$
\frac{d}{\ell_3} + \frac{2}{r_3} \leq 1 + \min\left\{\frac{d}{p},1\right\},
$$
$$
0 \leq \frac{1}{\ell_3} \leq 1 + \frac{1}{d} - \frac{1}{p}, \quad 0 < \frac{1}{r_3} \leq \min\left\{\frac{1}{2}, 1-\frac{1}{q} \right\},
$$
$$
0 < \frac{1}{r_3} < \frac{1}{2} \quad \textrm{if} \quad q = 2.
$$
If $0 < 1/r_3 = 1-1/q$ and $q \in (1,2)$,
$$
\frac{d}{\ell_3} + \frac{2}{r_3} < 1 + \min\left\{\frac{d}{p},1\right\}, \quad 0 \leq \frac{1}{\ell_3} \leq 1 + \frac{1}{d} - \frac{1}{p}.
$$
If $p = d \geq 2$, the pair $(\ell_3, r_3)$ with $\ell_3 \in [1,\infty]$ and $r_3 \in [1,\infty)$ satisfies
\begin{equation}
							\label{eq0626_01}
\frac{d}{\ell_3} + \frac{2}{r_3} < 2, \quad 0 \leq \frac{1}{\ell_3} < 1, \quad 0 < \frac{1}{r_3} \leq \min\left\{\frac{1}{2}, 1 - \frac{1}{q}\right\}.
\end{equation}
\end{condition}

\begin{condition}
							\label{cond0625_2}
The pair $(\ell_3,r_3)$ with $\ell_3 \in [1,\infty]$ and $r_3 \in [1,\infty)$ satisfies
\begin{equation}
							\label{eq0625_16}
\frac{d}{\ell_3} + \frac{2}{r_3} -1 \leq \min\left\{1,\frac{d}{p} + \frac{2}{q}, d+2-\left(\frac{d}{p} + \frac{2}{q}\right)\right\} =: \Phi(d,p,q),
\end{equation}
\begin{equation}
							\label{eq0625_18}
0 \leq \frac{1}{\ell_3} \leq \frac{1}{d} + \frac{1}{p}, \quad 0 < \frac{1}{r_3} \leq \frac{1}{2} + \frac{1}{q}.
\end{equation}
If $p = d/(d-1)< \infty$ or $r_3 = 1$, the pair $(\ell_3,r_3)$ with $\ell_3 \in [1,\infty]$ and $r_3 \in [1,\infty)$ satisfies
\begin{equation}
							\label{eq0625_04}
\frac{d}{\ell_3} + \frac{2}{r_3} -1 < \Phi(d,p,q), \quad 0 \leq \frac{1}{\ell_3} < \frac{1}{d} + \frac{1}{p}, \quad 0 < \frac{1}{r_3} < \frac{1}{2} + \frac{1}{q}.
\end{equation}
\end{condition}

\begin{remark}
In  the algebraic conditions \ref{condi0302_1}, we assume that $r_1, r_2, r_3 < \infty$.
However, when $d \geq 2$, one can consider the cases $r_1, r_2, r_3 = \infty$ in \ref{condi0302_1} if $\|a^i\|_{L_{\ell_1, r_1}(\Omega_T)}$, $\|b^i\|_{L_{\ell_2, r_2}(\Omega_T)}$ and $\|c\|_{L_{\ell_3, r_3}(\Omega_T)}$ are sufficiently small. See \cite[p.141, Remark 2.1]{MR0241822}.
\end{remark}

\begin{remark}
\label{rmk206}
In the case $p=q=2$, take $\kappa \in [1, \infty)$ if $d \ge 2$ and $\kappa \in [1, 2]$ if $d = 1$. Suppose $(\ell_3, r_3)$ satisfies Condition \ref{cond0625_2}. Then by setting
$$
\ell_1 = \ell_2 = 2\ell_3 = \frac{d\kappa}{\kappa - 1}, \quad r_1 = r_2 = 2r_3 = 2\kappa,
$$
we recover the conditions on the coefficients in \cite[Chapter III]{MR0241822} for the $L_2$-estimates.
To be precise, our conditions on $(\ell_i,r_i)$, $i=1,2,3$, when $p=q=2$ do not include the borderline cases with $\kappa = 1$.
However, these cases, that is,
$(\ell_i,r_i)=(\infty,2)$, $i=1,2$, and $(\ell_3,r_3) = (\infty,1)$ can also be covered by the estimates of equations in $V_{2,2}$-norm.
See Proposition \ref{190405@lem2} and \ref{lem0603_1}.
Regrading the condition on $f$ (see \eqref{40111}--\eqref{4011}) when $p=q=2$, we also recover the corresponding condition in \cite[Chapter III]{MR0241822} except the borderline case
$$
\frac{d}{p_1} + \frac{2}{q_1} = 2 + \frac{d}{2}, \quad q_1 = 1,
$$
which can also be included by the estimates in Propositions \ref{190405@lem2} and \ref{lem0603_1}.
\end{remark}

\subsection{The Dirichlet boundary value problem}

We first introduce an assumption for the unique solvability of equations without lower-order terms having the Dirichlet boundary condition.
Note that the summability of $f$ on the right-hand side of \eqref{eq0415_02} is not lower than that of $u$.
A few remarks are in order about the assumption.

\begin{assumption}[$p,q$]  \label{aomega} 
For $g=(g_1,\ldots,g_d)\in L_{p,q}(\Omega_T)^d$ and $f\in L_{p,q}(\Omega_T)$, 
there exists a unique solution $u\in \mathring{\mch}^1_{p,q}(\Omega_T)$ of 
\begin{equation}
							\label{eq0415_02}
 -u_t +D_i \left( a^{ij}(t,x) D_j u  \right)  = D_i g_i + f  \quad \textrm{in} \quad \Omega_T
\end{equation}
with $u(0, \cdot) = 0$ on $\Omega$.
Moreover, we have
$$
 \| u \|_{ W^{1,0}_{p,q} (\Omega_T) } 
\le K \left(   \| g\|_{ L_{p,q} (\Omega_T)  } +  \| f \|_{ L_{p,q} (\Omega_T) } \right),
$$
where $K > 0$ is independent of $g, f$, and $u$.
The same statement holds true when the pair $(p,q)$ is replaced by $(p',q')$, where $1/p + 1/p'=1/q+1/q'=1$ and the coefficients $a^{ij}$ are replaced by $a^{ji}$.
\end{assumption}

By solutions to \eqref{eq0415_02} with the initial zero condition, we mean so-called weak solutions satisfying an integral identity. 
See \eqref{eq0310_01} below.

\begin{remark}
As is well known, if $p=q=2$, Assumption \ref{aomega} is satisfied with the ellipticity condition \eqref{eq0402_02}.
For $p, q \in (1,\infty)$ with $(p,q) \neq (2,2)$, the $L_{p,q}$ solvability in Assumption \ref{aomega} depends on 
the smoothness of the main coefficients $a^{ij}(t,x)$ and the domain $\Omega$. For instance, Assumption \ref{aomega} is satisfied with an appropriate condition on $\Omega$ if $a^{ij}(t,x)$ are uniformly continuous in $x \in \mr^d$ or $a^{ij}(t,x)$ are in the class of vanishing mean oscillations (VMO) as functions of $x \in \mr^d$.
The regularity on $a^{ij}(t,x)$ can be further relaxed so that they can have sufficiently small mean oscillations or have no regularity conditions with respect to one spatial variable.
As to assumptions on $\Omega$, one can obtain the solvability in Assumption \ref{aomega} even when $\Omega$ is beyond the class of Lipschitz domains.
In fact, there are tremendous results about the unique solvability as in Assumption \ref{aomega} when the coefficients $a^{ij}$ and the boundaries are irregular.
Among those results, one can find related results in  \cite{MR1331981, MR2435520, MR3488249, MR2680179, MR3073000, MR2835999, MR2764911, MR3266252} and references therein.
In the same spirit as our assumption on the domain, instead of specifying a regularity assumption on $a^{ij}$ allowing the unique solvability in Assumption \ref{aomega} we show that one can have the unique solvability of equations with unbounded lower-order coefficients, provided that the leading coefficients $a^{ij}$ of the equations are regular enough to guarantee the solvability with the estimate in Assumption \ref{aomega}.
\end{remark}

\begin{remark}
One may obtain an a priori estimate as in Assumption \ref{aomega} for $u$ satisfying \eqref{eq0415_02} with $(p',q')$ via the duality argument whenever the solvability together with the estimate for solutions to \eqref{eq0415_02} is established for $(p,q)$.
However, to obtain the solvability result for $(p',q')$ using the a priori estimate and the method of continuity, one needs the solvability of, for instance, the heat equation.
Thus, if one can assure the solvability of the heat equation for $(p',q')$, it is enough to have the solvability and the estimate only for $(p,q)$ in Assumption \ref{aomega}.
However, since the main purpose of this paper is not to investigate whether or not the heat equation is solvable in $\mch^1_{p,q}(\Omega_T)$ or $\mch^1_{p',q'}(\Omega_T)$, 
we take the solvability of the equation \eqref{eq0415_02} for both $(p,q)$ and $(p',q')$ as a part of our assumptions.
\end{remark}

The following theorem is the main result when the equation has the Dirichlet boundary condition.

\begin{theorem}
\label{theo401}
Let $\Omega$ satisfy Assumption \ref{dom} $(\Omega)$.
Also let $1 < p,q<\infty$, $g=(g_1,\ldots,g_d)\in L_{p,q}(\Omega_T)^d$, and $f\in L_{p_1,q_1}(\Omega_T)$, where $(p_1,q_1)$ satisfies \eqref{40111}--\eqref{4011}.
Assume that the lower-order coefficients $a^i$, $b^i$, and $c$ satisfy \eqref{eq0402_03}.
Then, under Assumption \ref{aomega} $(p,q)$, 
there exists a unique $u\in \mathring{W}^{1,0}_{p,q}(\Omega_T)$ satisfying 
\begin{equation}		\label{402}
\mathcal{P}u  =D_i g_i +f \quad \text{in} \quad \Omega_T
\end{equation}
with $u(0, \cdot) = 0$ on $\Omega$ and 
\begin{equation}		
\label{403}
\|u\|_{W_{p, q}^{1,0}(\Omega_T)}  \le N \left(  \|g\|_{L_{p,q}(\Omega_T)}+   \|f\|_{L_{p_1,q_1}(\Omega_T)} \right),
\end{equation}
where $N=N(d, \delta, p,q,p_1,q_1, \ell_1, r_1, \ell_2,r_2,\ell_3,r_3, a^i, b^i, c, K, C, T)$.
In particular, $C$ represents a set of constants each of which is from Assumption \ref{dom} when $(p,q, p_0,q_0, p_k, q_k)$ is replaced with a set of numbers determined by $d$, $p$, $q$, $p_1$, $q_1$, $\ell_1$, $r_1$, $\ell_2$, $r_2$, $\ell_3$ and $r_3$.
\end{theorem}

We note that, throughout the paper, $p'$, $q'$, $p_1'$, and $q_1'$ stand for the H{$\ddot{\textrm{o}}$}lder conjugate exponents
of $p$, $q$, $p_1$, and $q_1$, respectively.

\subsubsection{Notions for solutions with the zero initial condition.}
\label{solution}
Let $1 < p, q < \infty$. A solution $u \in \mathring{W}_{p,q}^{1,0}(\Omega_T)$($\mathring{\mathcal{H}}_{p,q}^1(\Omega_T)$) satisfies \eqref{402} with $u(0, \cdot) = 0$ on $\Omega$ if
\begin{multline}
	\label{eq0310_01}
\int_0^T \int_{\Omega} \left( u \varphi_t  -a^{ij} D_ju D_i \varphi -a^i u D_i\varphi +b^i D_iu \, \varphi + c u \varphi   \right) dx\,dt
\\
= \int_0^T \int_{\Omega} \left( - g_i D_i \varphi + f\varphi \right) dx\,dt
\end{multline}
for all smooth test functions $\varphi$ defined on the closure of $\Omega_T$ and vanishing in a neighborhood of the lateral boundary and the upper base of $\Omega_T$. 

We note that if $u \in \mathring{W}_{p, q}^{1, 0}(\Omega_T)$ is a solution to \eqref{402} with $u(0, \cdot) = 0$ on $\Omega$, 
then $v(t, x) := u(t,x)\chi_{t \ge 0}$ is in $\mathring{W}_{p, q}^{1, 0}((-\infty, T) \times \Omega)$ and $\|v\|_{W_{p, q}^{1, 0}((-\infty, T)\times \Omega)}$ is equal to $\|u\|_{W_{p, q}^{1, 0}(\Omega_T)}$.
The same statement holds for solutions $u$ in $\mathring{\mathcal{H}}_{p, q}^1(\Omega_T)$ with $u(0, \cdot) = 0$ on $\Omega$.

\begin{remark} 
Theorem \ref{theo401} still holds with  $\sum_{k=1}^{n_0}f_k$ in place of $f$ in \eqref{402} if $f_k \in L_{p_k, q_k}(\Omega_T)$, $k=1,\ldots,n_0$, where each $(p_k, q_k)$ satisfies \eqref{40111}--\eqref{4011} with $(p_k, q_k)$ in place of $(p_1,q_1)$. In this case the constant $N$ additionally depends on $p_k$, $q_k$, and $n_0$.
One can just repeat the proof of Theorem \ref{theo401} with $\sum_{k=1}^{n_0} f_k$ instead of a single $f$.
\end{remark}

\begin{remark}
\label{counter2}
In the case that $1<p<2$ or $1<q<2$, if $u \in W_{p,q}^{1,0}(\Omega_T)$ is a solution of \eqref{402}, it is not likely that $u$ belongs to $V_{p,q}(\Omega_T)$.
Indeed, in the case $1<p=q<2$, the problem \eqref{402} is not necessarily solvable in the space $\mathring{V}_{p,p}(\Omega_T)$
even if $\mathcal{P}$ is the heat operator and the boundary of $\Omega$ is sufficiently smooth.
See \cite{MR2433518} for a counterexample.
On the other hand, as shown in Section \ref{Vspace}, if $p,q \in [2,\infty)$ and $u \in W_{p,q}^{1,0}(\Omega_T)$ is a solution of \eqref{402}, then $u \in V_{p,q}(\Omega_T)$.
However, even in this case $u$ does not necessarily belong to $\mathcal{H}_{p,q}^1(\Omega_T)$.
Indeed, let $f(t,x) = t^{-\alpha} \xi(x)$, where $1/q < \alpha < 1/q_1$ and $\xi \in C_0^\infty(\Omega)$.
Then $f \in L_{p_1,q_1}(\Omega_T)$, and by Proposition \ref{190405@lem2} there exists a unique solution $u \in V_{p,q}(\Omega_T)$ of \eqref{402} with $a^i=b^i=c=g_i=0$.
We see that
$$
u_t = D_i(a^{ij} D_j u) - f,
$$
where $a^{ij} D_j u \in L_{p,q}(\Omega_T)$, but $f \notin L_{p,q}(\Omega_T)$ by the choice of $\alpha$.
This shows that $u_t \notin \mathbb{H}_{p,q}^{-1}(\Omega_T)$.
In fact, one can show that there are no $g_i, h \in L_{p,q}(\Omega_T)$ satisfying $
u_t = D_i g_i + h$ in $\Omega_T$.
Hence, $u \notin \mathcal{H}_{p,q}^1(\Omega_T)$.
\end{remark}

\subsection{The conormal derivative boundary value problem}

Here is our assumption for the solvability of equations with the conormal derivative boundary condition.
 
\begin{assumption}[$p,q$]  \label{comega} 
For $g=(g_1,\ldots,g_d)\in L_{p,q}(\Omega_T)^d$ and $f\in L_{p,q}(\Omega_T)$, 
there exists a unique solution $u\in \mch^1_{p,q}(\Omega_T)$ of 
\begin{equation}
							\label{eq0515}
 \left\{
\begin{array}{lll}
\displaystyle -u_t + D_i\left( a^{ij}(t,x) D_j u  \right) = D_i g_i +f \quad &\displaystyle\textrm{in} \quad \Omega_T,
\\
\displaystyle \nu^ia^{ij}D_ju = \nu^ig_i \quad &\displaystyle\textrm{on} \quad (0, T) \times \partial\Omega. 
\end{array}
\right.
\end{equation}
with $u(0, \cdot) = 0$ on $\Omega$ and $\nu = (\nu^1, \ldots , \nu^d)$ is the outward normal vector on $\partial\Omega$. Moreover, we have 
$$
 \| u \|_{ W^{1,0}_{p,q} (\Omega_T) } \le K \left(   \| g\|_{ L_{p,q} (\Omega_T)  } +  \| f \|_{ L_{p,q} (\Omega_T) } \right),
$$
where $K > 0$ is independent of $g, f$, and $u$.
The same statement holds true when the pair $(p,q)$ is replaced by $(p',q')$, where $1/p + 1/p'=1/q+1/q'=1$ and the coefficients $a^{ij}$ are replaced by $a^{ji}$. 
\end{assumption}

The following theorem is the main result when the equation has the conormal derivative boundary condition.

\begin{theorem}
\label{conormal}
Under the same assumptions as in Theorem \ref{theo401} with Assumption \ref{aomega} ($p,q$) replaced with Assumption \ref{comega} ($p,q$), there exists a unique $u\in W_{p,q}^{1,0}(\Omega_T)$ satisfying
\begin{equation}		\label{602}
\left\{
\begin{array}{lll}
\displaystyle \mathcal{P}u = D_i g_i +f \quad &\displaystyle\textrm{in} \quad \Omega_T,
\\
\displaystyle \nu^ia^{ij}D_ju + \nu^ia^iu = \nu^ig_i \quad &\displaystyle\textrm{on} \quad (0, T) \times \partial\Omega. 
\end{array}
\right.
\end{equation}
with $u(0, \cdot) = 0$ on $\Omega$, where $\nu = (\nu^1, \ldots , \nu^d)$ is the outward normal vector on $\partial\Omega$. Moreover, we have
\begin{equation*}		
\|u\|_{W_{p, q}^{1, 0}(\Omega_T)}  \le N \left(  \|g\|_{L_{p,q}(\Omega_T)}+   \|f\|_{L_{p_1,q_1}(\Omega_T)} \right),
\end{equation*}
where $N=N(d, \delta, p, q, p_1, q_1, \ell_1, r_1, \ell_2,r_2,\ell_3,r_3, a^i, b^i, c, K, C, T)$.
Again, $C$ represents a set of constants each of which is from Assumption \ref{dom} when $(p,q,p_0,q_0, p_k, q_k)$ is replaced with a set of numbers determined by $d$, $p$, $q$, $p_1$, $q_1$, $\ell_1$, $r_1$, $\ell_2$, $r_2$, $\ell_3$ and $r_3$.
\end{theorem}

Here, we note that solutions $u \in W_{p, q}^{1,0}(\Omega_T)$ ($\mathcal{H}_{p, q}^1(\Omega_T)$) of \eqref{eq0515} and \eqref{602} with $u(0, \cdot) = 0$ on $\Omega$, 
are understood in an analogous way to the Dirichlet boundary value problem.
More precisely, we define solutions to the conormal boundary value problem using \eqref{eq0310_01} with test functions having (not necessarily) non-zero values on the lateral boundary of the parabolic domain.

\section{Proof of Theorem \ref{theo401} and \ref{conormal}}
\label{dirichlet}

In this section we prove Theorem \ref{theo401} and \ref{conormal}, our main result for the Dirichlet boundary condition case and the conormal derivative boundary condition case, respectively.

In the sequel we denote $\Omega_{T_1, T_2} = (T_1, T_2) \times \Omega$, where $-\infty < T_1 <T_2 < \infty$. Thus, $\Omega_T = \Omega_{0,T}$.
Also, as mentioned in Assumption \ref{dom}, we may denote by $C_k$ the constant from Assumption \ref{dom} ($\Omega$) with $(p,q, p_0,q_0, p_k,q_k)$ or $(p', q', p_0',q_0', p'_k, q'_k)$.

\begin{lemma}		
\label{190405@lem1}
Let $p,q\in (1, \infty)$ and $m$ be a positive integer and $0 \le t_1 < t_2 \le T$.
Also let $(p_k, q_k)$, $k=1,\ldots,m$, satisfy \eqref{40111}--\eqref{4011} in place of $(p_1,q_1)$.
Suppose that Assumption \ref{dom} ($\Omega$) and Assumption \ref{aomega} $(p,q)$ hold. Then, for $u\in \mathring{W}^{1,0}_{p,q}(\Omega_{t_1, t_2})$ with $u(t_1, \cdot) = 0$ on $\Omega$ satisfying
\begin{equation}		\label{190405@eq1}
-u_t +D_i ( a^{ij}(t,x) D_j u)   = D_i g_i + \sum_{k=1}^m f_k \qquad  \textrm{in } \quad   \Omega_{t_1, t_2},
\end{equation}
where $g=(g_1,\ldots,g_d) \in L_{p,q}(\Omega_{t_1, t_2})^d$ and $f_k \in L_{p_k,q_k}(\Omega_{t_1, t_2})$, we have  
\begin{equation*}
\| u \|_{W_{p, q}^{1, 0}(\Omega_{t_1, t_2}) }   \leq N\left( \| g\|_{ L_{p,q} (\Omega_{t_1, t_2}) }  +  \sum_{k=1}^{m}   \| f_k \|_{ L_{p_k,q_k} (\Omega_{t_1, t_2}) }\right), 
\end{equation*}
where $N=N(d,\delta, m, K, C_k)$, but independent of $t_1$ or $t_2$.
\end{lemma} 
 
\begin{proof}
Note that $p_k' \in [p',\infty]$, $q_k' \in [q',\infty]$, and
the pairs $(p',q')$ and $(p_k',q_k')$ satisfy \eqref{eq0602_03} if $q_k'=q'$, and \eqref{eq0602_01} and \eqref{eq0602_02} if $q_k' > q'$ in place of $(p,q)$ and $(p_0,q_0)$, respectively.
That is,
\begin{enumerate}
\item[(i)] if $q_k' = q'$,
$$
\frac{d}{p'} \leq 1 + \frac{d}{p_k'}, \quad (p',p_k') \neq \left( d(\geq 2), \infty\right),
$$
\item[(ii)] if $q_k' > q'$,
$$
\frac{d}{p'} + \frac{2}{q'} \leq 1 + \frac{d}{p_k'} + \frac{2}{q_k'},
$$
and
$$
1 < q' < q_k' < \infty \quad \text{if} \quad \frac{d}{p'} + \frac{2}{q'} = 1 + \frac{d}{p_k'} + \frac{2}{q_k'}.
$$
\end{enumerate}
Hence, we are able to use the embedding inequality \eqref{imb} when both $(p,q)$ and $(p_k,q_k)$ in Assumption \ref{dom} are given by $(p',q')$, and $(p_0,q_0)$ in Assumption \ref{dom} is given by $(p_k',q_k')$. See Remark \ref{rem0602_1}.
Recall that $u\chi_{t \geq t_1} \in W_{p, q}^{1, 0}((-\infty, t_2) \times \Omega)$.
Thus if we extend $u = 0$ for $t \le t_1$ and extend $u$ appropriately for $t \geq t_2$, then $\bar{u}$, the extension of $u$, belongs to $\mathring{W}_{p, q}^{1, 0}(\Omega_T)$. Morerover, $\bar{u}$ satisfies
$$
-\bar{u}_t + D_i(a^{ij}(t, x)D_j\bar{u}) = D_i\bar{g}_i + \sum_{k=1}^{m}\bar{f}_k \quad \textrm{in} \quad \Omega_T,
$$
where $\bar{g}_i$ and $\bar{f}_k$ are appropriate extensions on $\Omega_T$ of $g_i$ and $f_k$ so that $\bar{g}_i = g_i\chi_{t \geq t_1}$ and $\bar{f}_k = f_k\chi_{t \geq t_1}$ on $(0, t_2) \times \Omega$, and the norms $\|\bar{g}\|_{L_{p,q}(\Omega_T)}$ and $\|\bar{f}_k\|_{L_{p_k,q_k}(\Omega_T)}$ are comparable to $\|g\|_{L_{p,q}(\Omega_{t_1,t_2})}$ and $\|f_k\|_{L_{p_k,q_k}(\Omega_{t_1,t_2})}$, respectively.

To prove the estimate, we use a duality argument.
Let $\psi\in C^\infty_0(\Omega_{t_1, t_2})$ and $\varphi=(\varphi_1,\ldots,\varphi_d)\in C^{\infty}_0( \Omega_{t_1, t_2} )^d $. Note that $\psi\in C^\infty_0(\Omega_T)$ and $\varphi \in C^{\infty}_0( \Omega_T )^d$ by extending $\psi = \varphi = 0$ for $t < t_1$ and $t > t_2$.
By Assumption \ref{aomega} $(p,q)$, there exists a unique $w \in \mathring{\mch}^1_{p', q'}(\Omega_T)$ satisfying $w(T, \cdot) = 0$ on $\Omega$,
\begin{equation*}
w_t + D_j(a^{ij}(t,x) D_i w)  = D_i \varphi_i -  \psi  \qquad \textrm{in } \quad \Omega_T
\end{equation*}
and
\begin{equation}		\label{190409@A1}
\| w \|_{ L_{p',q'} (\Omega_T) } + \|  Dw \|_{ L_{p',q'} (\Omega_T) } \le K \left( \| \varphi \|_{ L_{p',q'} (\Omega_T) }+\|\psi\|_{L_{p',q'}(\Omega_T)} \right).
\end{equation}
Upon writing that
$$
w_t=D_j G_j +F        \qquad \text{in } \quad  \Omega_T,
$$
where 
$G_j=-a^{ij}D_i w+\varphi_j$ and $F= - \psi$, the inequality \eqref{190409@A1} shows that
$$
\|G_j\|_{L_{p',q'}(\Omega_T)}+\|F\|_{L_{p',q'}(\Omega_T)} \le N_0\left( \| \varphi \|_{ L_{p',q'} (\Omega_T) }+\|\psi\|_{L_{p',q'}(\Omega_T)} \right),
$$
where $N_0=N_0(d,\delta, K)$.
From this inequality along with the definition of $\mathbb{H}^{-1}_{p,q}(\Omega_T)$, \eqref{190409@A1}, and the embedding inequality \eqref{imb} in Assumption \ref{dom} ($\Omega$), we have 
\begin{equation}		\label{190323@A3}
\|w\|_{L_{p'_k,q_k'}(\Omega_T)}\le  N \left( \| \varphi \|_{ L_{p',q'} (\Omega_{t_1, t_2}) }+\|\psi\|_{L_{p',q'}(\Omega_{t_1, t_2})} \right),
\end{equation}
where $N=N(d,\delta, K, C_k)$.
Recall that $C_k$ is the constant from Assumption \ref{dom} when both $(p,q)$ and $(p_k,q_k)$ are given by $(p',q'
)$, and $(p_0,q_0)$ is given by $(p_k',q_k')$.
Therefore, using H\"older's inequality, \eqref{190409@A1}, \eqref{190323@A3}, and the fact that
$$
\int_{\Omega_{t_1, t_2}} u \psi\,dx\,dt+\int_{\Omega_{t_1, t_2}}  D u\cdot \varphi \, dx\,dt =
$$
$$
\int_{\Omega_T} \bar{u} \psi\,dx\,dt+\int_{\Omega_T}  D \bar{u}\cdot \varphi \, dx\,dt = \int_{\Omega_T} \bar{g}_i D_iw \, dx\,dt    - \sum_{k=1}^m\int_{ \Omega_T } \bar{f}_k w \, dx\,dt,
$$
we obtain 
\begin{align*}
&\left| \int_{\Omega_{t_1, t_2}} u \psi\,dx\,dt+\int_{\Omega_{t_1, t_2}}  D u\cdot \varphi \, dx\, dt\right|
\\
&\le N  \left(\|\bar{g}\|_{L_{p,q}(\Omega_T)}+\sum_{k=1}^m   \|\bar{f}_k\|_{L_{p_k,q_k}(\Omega_T)}\right)\||\varphi| + |\psi|\|_{ L_{p',q'} (\Omega_{t_1, t_2}) }
\\
&\le N \left(\|g\|_{L_{p,q}(\Omega_{t_1, t_2})}+\sum_{k=1}^m \|f_k\|_{L_{p_k,q_k}(\Omega_{t_1, t_2})}\right)\| |\varphi| + |\psi| \|_{ L_{p',q'} (\Omega_{t_1, t_2})}.
\end{align*}
Since the above inequality holds for any $\varphi\in C^\infty_0(\Omega_{t_1, t_2})^d$ and $\psi\in C^\infty_0(\Omega_{t_1, t_2})$, by the duality we get the desired estimate.
The lemma is proved.
 \end{proof}
 
In the following we prove a version of Theorem \ref{theo401} when the lower-order coefficients are all zero.

\begin{proposition}
	\label{prop200422_01}
Let $\Omega$ satisfy Assumption \ref{dom} $(\Omega)$. Also let $1 < p,q<\infty$, $g=(g_1,\ldots,g_d)\in L_{p,q}(\Omega_T)^d$, and $f_k \in L_{p_k,q_k}(\Omega_T)$, $k = 1, \ldots,  m$, where $(p_k,q_k)$ satisfies \eqref{40111}--\eqref{4011}. 
Then, under Assumption \ref{aomega} $(p,q)$, 
there exists a unique $u\in \mathring{W}_{p, q}^{1, 0}(\Omega_T)$ satisfying \eqref{190405@eq1} with $u(0, \cdot) = 0$ on $\Omega$ and 
$$
\|u\|_{L_{p, q}(\Omega_T)} + \|Du\|_{L_{p, q}(\Omega_T)}  \le N \left( \|g\|_{L_{p,q}(\Omega_T)} + \sum_{k=1}^m \|f_k\|_{L_{p_1,q_1}(\Omega_T)} \right),
$$
where $N=N(d,\delta, p,q,p_k,q_k,m,K,C_k,T)$.
\end{proposition}

\begin{proof} 
The proposition follows from Assumption \ref{aomega} $(p,q)$ along with the estimate in Lemma \ref{190405@lem1} with $f_k^n \in L_{p,q}(\Omega_T)$ such that $f_k^n \to f_k$ in $L_{p_k,q_k}(\Omega_T)$ as $n \to \infty$.
Indeed, by Assumption \ref{aomega} $(p,q)$, there exists unique solution $u^n \in \mathring{\mathcal{H}}_{p, q}^1(\Omega_T)$ of the equation
$$
-u^n_t + D_i\left( a^{ij}D_ju^n \right) = D_ig_i + \sum_{k=1}^m f^n_k \quad \textrm{in} \quad \Omega_{T}
$$
with $u^n(0, \cdot) = 0$ on $\Omega$ where $f^n_k \in C_0^{\infty}(\Omega_T)$ and $f^n_k \to f_k$ in $L_{p_k, q_k}(\Omega_T)$, $k = 1, \ldots ,m$.
Then by applying Lemma \ref{190405@lem1} with $(t_1, t_2) = (0, T)$ to $u^n$, $n = 1, 2, \ldots$, we obtain the desired result.
\end{proof}

Now we are ready to prove Theorem \ref{theo401}.
 
 \begin{proof}[Proof of Theorem \ref{theo401}]
By Proposition \ref{prop200422_01}, we have the unique solvability and estimate for the equation when the lower-order coefficients are all zero. 

To deal with the equation with lower-order coefficients, throughout the proof, we fix
$$
(p_2, q_2) = \left(\frac{\ell_2p}{\ell_2 + p}, \frac{r_2q}{r_2 + q} \right).
$$
We also fix $(p_3,q_3)$ as follows.
If $(\ell_3,r_3)$ satisfies Condition \ref{cond0625_1}, we set
$$
\frac{1}{p_3} = \frac{1}{\ell_3}+\frac{1}{p}-\frac{1}{\alpha}, \quad \frac{1}{q_3} = \frac{1}{r_3}+\frac{1}{q},
$$
where $\alpha$ satisfies
\begin{equation}
							\label{eq0625_01}
\max\left\{0, \frac{d}{\ell_3} + \frac{d}{p} - d, \frac{d}{\ell_3} + \frac{2}{r_3}- 1 \right\} \leq \frac{d}{\alpha} \leq \min\left\{\frac{d}{\ell_3}, \frac{d}{p},1\right\},
\end{equation}
and, if $1/r_3=1-1/q$ and $q \in (1,2)$, $\alpha$  satisfies \eqref{eq0625_01} as well as
\begin{equation}
						\label{eq0625_13}
\frac{d}{\ell_3} + \frac{2}{r_3}- 1 < \frac{d}{\alpha}.
\end{equation}
If $p = d \geq 2$, we additionally require that $\alpha > d$, which is possible by \eqref{eq0626_01}.
We see that $(p_3,q_3)$ satisfies \eqref{40111}--\eqref{4011} in place of $(p_1,q_1)$.
In particular, the first condition in \eqref{40111} is satisfied because $r_3 < \infty$, that is, $q_3 \neq q$.
Thanks to \eqref{eq0625_13}, the third condition \eqref{4011} also holds.
If we denote
$$
p_0 = \frac{\ell_3 p_3}{\ell_3 - p_3}, \quad q_0 = \frac{r_3q_3}{r_3 - q_3},
$$
then $(p_0,q_0)$ satisfies (i) in Assumption \ref{dom}.
Precisely, the conditions in \eqref{eq0602_03} follow from $d/\alpha \leq 1$ and $p_3 < \ell_3$ for $p=d \geq 2$.

If $(\ell_3,r_3)$ satisfies Condition \ref{cond0625_2}, we set
$$
\frac{1}{p_3} = \frac{1}{\ell_3}+\frac{1}{p}-\frac{1}{\alpha},
\quad \frac{1}{q_3} = \frac{1}{r_3}+\frac{1}{q}-\frac{1}{\beta},
$$
where $\alpha$ and $\beta$ satisfy
\begin{equation}
							\label{eq0625_17}
\max\left\{\frac{d}{\ell_3} + \frac{d}{p} - d, 0\right\} \leq \frac{d}{\alpha} \leq \min\left\{\frac{d}{\ell_3}, \frac{d}{p} \right\},
\end{equation}
\begin{equation}
							\label{eq0625_02}
\max\left\{\frac{2}{r_3} + \frac{2}{q} - 2, 0 \right\} \leq \frac{2}{\beta} \leq \min\left\{ \frac{2}{r_3}, \frac{2}{q} \right\},
\end{equation}
\begin{equation}
							\label{eq0625_03}
\frac{2}{r_3} + \frac{2}{q} - 2 < \frac{2}{\beta} < \frac{2}{q} \quad \text{if} \quad r_3 > 1,
\end{equation}
and
\begin{equation}
							\label{eq0625_14}
\frac{d}{\ell_3} + \frac{2}{r_3}- 1 \leq \frac{d}{\alpha} + \frac{2}{\beta} \leq \min\left\{1, \frac{d}{p} + \frac{2}{q}\right\}.
\end{equation}
If $p = d/(d-1) < \infty$ or $r_3 = 1$, instead of \eqref{eq0625_14}, we take $\alpha$ and $\beta$ satisfying
\begin{equation}
							\label{eq0625_05}
\frac{d}{\ell_3} + \frac{2}{r_3}- 1 < \frac{d}{\alpha} + \frac{2}{\beta} < \min\left\{1,\frac{d}{p} + \frac{2}{q}\right\}.
\end{equation}
If $p = d \geq 2$, we additionally require that $\alpha > d$, which is possible by the following reason.
The choice $\alpha = d = p$ (thus, $1/\beta = 0$) is unavoidable only when $\alpha = d$ is the only choice from \eqref{eq0625_17} with $d/\ell_3 \geq 1$.
Indeed, because $2/r_3 > 0$, there is a non-empty intersection between $(0,2/r_3)$ and the interval in \eqref{eq0625_03}.
Thus, one can choose an appropriate pair $(\alpha, \beta)$ with $\alpha > d$ satisfying \eqref{eq0625_14} if $d/\ell_3 + 1 - d < 1$ in \eqref{eq0625_17}.
When $\alpha =d$ is the only choice from \eqref{eq0625_17}, that is, $d/\ell_3 + 1 - d = 1$, we have $\ell_3 = 1$, which is impossible by the conditions $d/\ell_3 + 2/r_3 \leq 2$, $d \geq 2$, and $r_3 < \infty$.

Note that we need \eqref{eq0625_16}, in particular,
$$
\frac{d}{\ell_3} + \frac{2}{r_3} \leq 1 + d+2 - \left(\frac{d}{p} + \frac{2}{q}\right)
$$
to have a non-empty interval for $d/\alpha+2/\beta$ in \eqref{eq0625_14} when
$$
d+2 - (d/p+2/q) < 1 \leq d/p+2/q.
$$
Also note that, to have again a non-empty interval for $d/\alpha+2/\beta$ in \eqref{eq0625_14}, we necessarily need
$$
\frac{d}{\ell_3} + \frac{2}{r_3} - 1 \leq \min\left\{\frac{d}{\ell_3}, \frac{d}{p} \right\} + \min\left\{\frac{2}{r_3}, \frac{2}{q} \right\}.
$$
In particular, we have
$$
\frac{d}{\ell_3} + \frac{2}{r_3} - 1 \leq \frac{d}{p} + \frac{2}{r_3}, \quad \frac{d}{\ell_3} + \frac{2}{r_3} - 1 \leq \frac{d}{\ell_3} + \frac{2}{q},
$$
which explains the necessity of the inequalities in \eqref{eq0625_18}.
Similarly, to have a non-empty half open interval for $d/\alpha+2/\beta$ in \eqref{eq0625_05} we have \eqref{eq0625_04} for $p=d/(d-1)$ or $r_3=1$.
Observe that $(p_3,q_3)$ satisfies \eqref{40111}--\eqref{4011} in place of $(p_1,q_1)$.
In particular, when $p = d/(d-1)$, it follows from \eqref{eq0625_05} that $(p_3, q_3) \neq (1, q)$.
Indeed, if $(p_3, q_3) = (1, q)$, that is,
$$
\frac{1}{\ell_3} = 1 - \frac{1}{p} + \frac{1}{\alpha}, \quad \frac{1}{r_3} = \frac{1}{\beta} \quad \Rightarrow  \quad \frac{d}{\ell_3} + \frac{2}{r_3} - 1 = \frac{d}{\alpha} + \frac{2}{\beta},
$$
which is prohibited by \eqref{eq0625_05}.
The condition \eqref{401} is satisfied by \eqref{eq0625_14}.
If $r_3 > 1$, then $q_3 > 1$ by \eqref{eq0625_03}.
If $r_3 = 1$, then $\beta = q$ by \eqref{eq0625_02}, which implies $q_3 = 1$.
In this case by \eqref{eq0625_05} we have
$$
\frac{d}{p_3} + \frac{2}{q_3} < 1 + \frac{d}{p} + \frac{2}{q}.
$$
Hence, \eqref{4011} is also satisfied.
As above, if we denote
$$
p_0 = \frac{\ell_3 p_3}{\ell_3 - p_3}, \quad q_0 = \frac{r_3q_3}{r_3 - q_3},
$$
then $(p_0,q_0)$ satisfies (i) or (ii) in Assumption \ref{dom}.
Indeed, if $q_0 = q$, that is, $1/\beta = 0$, then by \eqref{eq0625_14} we have $d/\alpha \leq 1$, which guarantees
$$
\frac{d}{p} \leq 1 + \frac{d}{p_0}.
$$
The second condition in \eqref{eq0602_03} holds because $p_0 <\infty$ by the choice of $\alpha > d$ when $p = d$.
If $q_0 > q$, the conditions in \eqref{eq0602_01} and \eqref{eq0602_02} are satisfied by the choices of $\alpha$ and $\beta$ in \eqref{eq0625_03}, \eqref{eq0625_14}, and \eqref{eq0625_05}.
The conditions in \eqref{eq0603_02} and \eqref{eq0603_03} readily follow from $d/\ell_3 + 2/r_3 \leq 2$.
In particular, we have $q_0 < \infty$ by the choice of $\beta$ in \eqref{eq0625_03} for $r_3 > 1$.
If $r_3 = 1$, then as above, $\beta = q$ and $q_3 = 1$, which means that $q_0 = \infty$.

We summarize the properties of the pairs given above as follows. The pairs $(p_2,q_2)$ and $(p_3,q_3)$ satisfy \eqref{40111}--\eqref{4011} as $(p_1,q_1)$ does. 
The pairs
$$
\left(\frac{\ell_1 p}{\ell_1 - p}, \frac{r_1q}{r_1-q}\right) \quad \textrm{and} \quad \left(\frac{\ell_3p_3}{\ell_3 -p_3}, \frac{r_3q_3}{r_3 - q_3}\right)
$$
satisfy the conditions in Assumption \ref{dom} in place of $(p_0,q_0)$.
In particular,
$$
\frac{r_1q}{r_1-q} \in (q,\infty] \quad \text{and} \quad \frac{r_3q_3}{r_3 - q_3} \in [q, \infty).
$$
Moreover, 
the triples
\begin{equation}
							\label{eq0603_05}
\left(\frac{\ell_1 p}{\ell_1 - p}, \frac{r_1q}{r_1-q}\right), \,\, (p,q), \,\, (p_k,q_k), \quad k = 1,2,3,
\end{equation}
and
\begin{equation}
							\label{eq0603_06}
\left(\frac{\ell_3p_3}{\ell_3 -p_3}, \frac{r_3q_3}{r_3 - q_3}\right),\,\, (p,q), \,\, (p_k,q_k), \quad k = 1,2,3,
\end{equation}
satisfy the conditions in Assumption \ref{dom} in place of the triple $(p_0,q_0)$, $(p,q)$, and $(p_k,q_k)$ there.
In particular, to check the conditions in Assumption \ref{dom} for $(p_k,q_k)$, we use the fact that $(p_k,q_k)$, $k=1,2,3$, satisfy \eqref{40111}--\eqref{4011}.

Set $\mathfrak{H}(\Omega_T)$ to be the collection of functions $u \in W_{p,q}^{1,0}(\Omega_T)$ such that
$$
u_t = D_ig_i + \sum_{k=1}^3 f_k \quad \textrm{in} \quad \Omega_T
$$
in the distribution sense,
where $f_k \in L_{p_k, q_k}(\Omega_T)$ for $k = 1, 2, 3$.
Note that $\mathfrak{H}(\Omega_T)$ is a Banach space with the norm
$$
\|v\|_{\mathfrak{H}(\Omega_T)} = \| v \|_{W_{p, q}^{1, 0}(\Omega_T)} + \inf\left\{ \|g\|_{L_{p,q}(\Omega_T)}+\sum_{k=1}^3\|f_k\|_{L_{p_k,q_k}(\Omega_T)}: v_t=D_i g_i + \sum_{k=1}^3 f_k \right\}.
$$
Then, for $u \in \mathfrak{H}(\Omega_T)$, 
since the triples in \eqref{eq0603_05} and \eqref{eq0603_06} satisfy the conditions in Assumption \ref{dom}, we have
\begin{multline*}
\|u\|_{L_{\frac{\ell_1 p}{\ell_1 - p}, \frac{r_1q}{r_1-q}}(\Omega_T)} + \|u\|_{L_{\frac{\ell_3p_3}{\ell_3 -p_3}, \frac{r_3q_3}{r_3 - q_3}}(\Omega_T)}
\\
\leq C \left( \|u\|_{W_{p,q}^{1,0}(\Omega_T)} + \|g\|_{L_{p,q}(\Omega_T)} + \sum_{k=1}^3 \|f_k\|_{L_{p_k,q_k}(\Omega_T)}\right)
\end{multline*}
whenever $u_t = D_i g_i + \sum_{k=1}^3f_k$ in $\Omega_T$ in the distribution sense.
Thus, by the definition of the norm of $\mathfrak{H}(\Omega_T)$, we can say that
\begin{equation}
							\label{eq0603_07}
\|u\|_{L_{\frac{\ell_1 p}{\ell_1 - p}, \frac{r_1q}{r_1-q}}(\Omega_T)} + \|u\|_{L_{\frac{\ell_3p_3}{\ell_3 -p_3}, \frac{r_3q_3}{r_3 - q_3}}(\Omega_T)}
\leq C \|u\|_{\mathfrak{H}(\Omega_T)}.
\end{equation}

We now prove the a priori estimate \eqref{403} for $u \in \mathring{W}_{p,q}^{1,0}(\Omega_T)$ satisfying \eqref{402} under the assumption that $u \in \mathfrak{H}(\Omega_T)$.
Clearly, in the case without lower-order terms we just have solved, the solution belongs to $\mathfrak{H}(\Omega_T)$.

Since $u\in \mathfrak{H}(\Omega_T)$ is a solution of \eqref{402},
it holds that 
\begin{equation}
							\label{eq0501_01}
-u_t+D_i (a^{ij}D_j u) =D_i \tilde{g}_i +f-b^i D_i u-cu \qquad \text{in } \quad \Omega_T,
\end{equation}
where $\tilde{g}_i=g_i-a^i u$.
We estimate the terms on the right-hand side of the identity \eqref{eq0501_01} as follows.

\begin{enumerate}[(i)]
\item
Estimate of $\tilde{g}_i$: By the triangle and H\"older's inequalities, we have 
$$
\begin{aligned}
\|\tilde{g}_i\|_{L_{p,q}(\Omega_T)}
&\le \|g_i\|_{L_{p,q}(\Omega_T)}+\|a^i u \|_{L_{p,q}(\Omega_T)} \\
&\le \|g_i\|_{L_{p,q}(\Omega_T)}+\|a^i \|_{L_{\ell_1,r_1}(\Omega_T)}\|u\|_{L_{\frac{\ell_1 p}{\ell_1-p}, \frac{ r_1q }{r_1-q}}(\Omega_T)}.  \\
\end{aligned}
$$
From this and \eqref{eq0603_07} it follows that
$$
\|\tilde{g}_i\|_{L_{p,q}(\Omega_T)}  \le \|g_i\|_{L_{p,q}(\Omega_T)} + C \|a^i \|_{L_{\ell_1,r_1}(\Omega_T)} \|u\|_{\mathfrak{H}(\Omega_T)}.
$$

\item
Estimate of $b^i D_i u$:
By H\"older's inequality, we have 
$$
\|b^i D_i u\|_{L_{{p}_2,{q}_2}(\Omega_T)}\le \|b^i\|_{L_{\ell_2, r_2}(\Omega_T)} \|D_i u\|_{L_{p,q}(\Omega_T)}.
$$

\item
Estimate of $cu$:
By H\"older's inequality and \eqref{eq0603_07}, we have 
$$
\begin{aligned}
\|cu\|_{L_{{p}_3, {q}_3}(\Omega_T)}
&\le \|c\|_{L_{\ell_3,r_3}(\Omega_T)} \|u\|_{L_{\frac{\ell_3p_3}{\ell_3 -p_3}, \frac{r_3q_3}{r_3 - q_3}}(\Omega_T)}\\
&\le C \|c\|_{L_{\ell_3,r_3}(\Omega_T)} \| u \|_{\mathfrak{H}(\Omega_T)}.
\end{aligned}
$$
\end{enumerate} 

Now we are ready to prove the estimate \eqref{403}.
By applying Lemma \ref{190405@lem1} to the equation \eqref{eq0501_01} with $(t_1, t_2) = (0, T)$, we have
\begin{equation}		\label{190412@eq33}
\begin{aligned}
\|u\|_{\mathfrak{H}(\Omega_T)} \le N' \left( \|\tilde{g}\|_{L_{p,q}(\Omega_T)} + \|f\|_{L_{p_1,q_1}(\Omega_T)} +  \|b^i D_i u\|_{L_{p_2,q_2}(\Omega_T)}+ \|cu\|_{L_{p_3,q_3}(\Omega_T)} \right),
\end{aligned}
\end{equation}
where $N' = N'(d, \delta, p, q, p_1, q_1, \ell_k, r_k, C, K)$.
Note that $C$ represents the constants from Assumption \ref{dom} $(\Omega)$ associated with the parameters in \eqref{eq0603_05} and \eqref{eq0603_06}.
Using the estimates derived in (i)--(iii), we see that the right-hand side of the inequality \eqref{190412@eq33} is bounded by 
$$
\begin{gathered}
N' \|g\|_{L_{p,q}(\Omega_T)}+N' \|f\|_{L_{p_1,q_1}(\Omega_T)} + N' \|b^i\|_{L_{\ell_2, r_2}(\Omega_T)}  \|Du\|_{L_{p,q}(\Omega_T)}\\
+N' \| u \|_{\mathfrak{H}(\Omega_T)} \left( C_1  \|a^i \|_{L_{\ell_1,r_1}(\Omega_T)} + C_3  \|c\|_{L_{\ell_3,r_3}(\Omega_T)} \right)  .
\end{gathered}
$$
Then we have from \eqref{190412@eq33} that
\begin{equation}
\label{308}
\begin{aligned}
\|u\|_{\mathfrak{H}(\Omega_T)} \le &N' \|g\|_{L_{p,q}(\Omega_T)} + N' \|f\|_{L_{p_1,q_1}(\Omega_T)}
\\
&+ N'' \|u\|_{\mathfrak{H}(\Omega_T)},
\end{aligned}
\end{equation}
where 
$$
N''=  N'  N_1  \|a^i \|_{L_{\ell_1,r_1}(\Omega_T)} + N' \|b^i\|_{L_{\ell_2, r_2}(\Omega_T)}   + N' N_3  \|c\|_{L_{\ell_3,r_3}(\Omega_T)}.
$$
If $N'' < 1$, then with the use of \eqref{308} we obtain the a priori estimate \eqref{403}. 
If not, we apply the method based on splitting the interval $[0, T]$ as follows. 
Note that 
$$
N'' \le N_0  \left( \int_0^T  \left( \| a^i (t, \cdot) \|_{L_{\ell_1}(\Omega)} ^{r_1} + \| b^i (t, \cdot ) \|_{L_{\ell_2} (\Omega)}^{r_2} + \| c(t, \cdot) \|_{L_{\ell_3}(\Omega)}^{r_3} \right) \, dt \right)^{\frac{1}{r_0}}
$$
for some $r_0 \in \{ r_1, r_2, r_3 \}$ and $N_0=N_0(d,\delta, p, q, p_1, q_1, \ell_1, r_1, \ell_2,r_2,\ell_3, r_3, K, C, T) $.
Denote
$$
\mu(s_1,s_2) :=  N_0  \left( \int_{s_1}^{s_2}  \left( \| a^i (t, \cdot) \|_{L_{\ell_1}(\Omega)} ^{r_1} + \| b^i (t, \cdot ) \|_{L_{\ell_2} (\Omega)}^{r_2} + \| c(t, \cdot) \|_{L_{\ell_3}(\Omega)}^{r_3} \right) \, dt \right)^{\frac{1}{r_0}}.
$$
We divide the interval $[0, T]$ into a finite number of subintervals 
$[t_0, t_1]$, $[t_1, t_2]$,  $ \ldots$, $[t_{\tilde{n}-1}, t_{\tilde{n}}]$ such that $t_0=0$, $t_{\tilde{n}}=T$, and 
$$
\frac{1}{2} \left( \frac{1}{4N_0} \right)^{r_0} \le  
\int_{t_{k-1}}^{t_k}  \left( \| a^i (t, \cdot) \|_{L_{\ell_1}(\Omega)} ^{r_1} + \| b^i (t, \cdot ) \|_{L_{\ell_2} (\Omega)}^{r_2} + \| c(t, \cdot) \|_{L_{\ell_3}(\Omega)}^{r_3} \right) \, dt
$$
$$
\le \frac{1}{2} \left( \frac{1}{2N_0} \right)^{r_0}
$$
for $k=1,2, \cdots \tilde{n}$.
Then
$$
 \frac{1}{2^{1/r_0} \cdot 4} \le \mu(t_{k-1}, t_k) \le \frac{1}{2^{1/r_0} \cdot 2}.
$$
Note that 
$$ \begin{aligned}
\frac{ \tilde{n} } {2}  \left( \frac{1}{4} \right)^{r_0} & \le \, \sum_{k=1}^{\tilde{n}} \left( \mu(t_{k-1}, t_k) \right)^{r_0}\\
&  = \, N_0^{r_0}   \int_{0}^{T}  \left( \| a^i (t, \cdot) \|_{L_{\ell_1}(\Omega)} ^{r_1} + \| b^i (t, \cdot ) \|_{L_{\ell_2} (\Omega)}^{r_2} + \| c(t, \cdot) \|_{L_{\ell_3}(\Omega)}^{r_3} \right) \, dt.
\end{aligned}$$
That is,
$$
\tilde{n} \le 2 (4 N_0)^{r_0} \int_{0}^{T}  \left( \| a^i (t, \cdot) \|_{L_{\ell_1}(\Omega)} ^{r_1} + \| b^i (t, \cdot ) \|_{L_{\ell_2} (\Omega)}^{r_2} + \| c(t, \cdot) \|_{L_{\ell_3}(\Omega)}^{r_3} \right) \, dt.
$$

Now we find finitely many points $\{ 0= s_0, s_1, \cdots,  s_{n_1}=T \}$ such that $s_k -s_{k-1}= l $ for $k=1,2, \cdots, n_1 - 1$ and $s_{n_1} - s_{n_1 - 1} < l$, where 
$$
l := \min \{ t_k - t_{k-1} :  k=1, 2, \cdots, \tilde{n} \}.
$$
Since, for each $n = 1, 2, \ldots, n_1$, there exists $j \in \{0,1, \cdots, \tilde{n}-2\}$ such that $[s_{n-1}, s_n] \subset [t_{j}, t_{j+2}]$, we see that
$$
\int_{s_{n-1}}^{s_n}  \left( \| a^i (t, \cdot) \|_{L_{\ell_1}(\Omega)} ^{r_1} + \| b^i (t, \cdot ) \|_{L_{\ell_2} (\Omega)}^{r_2} + \| c(t, \cdot) \|_{L_{\ell_3}(\Omega)}^{r_3} \right) \, dt
\le  \left( \frac{1}{2 N_0} \right)^{r_0}.
$$
Thus, $ \mu(s_{n-1},s_n) \le \frac{1}{2}$.
For convenience, we substitute $t_k$ for $s_k$.

Note that the above steps can be repeated with $\Omega_{t_0,t_1}$ (i.e., $\Omega_{0,t_1}$) in place of $\Omega_T$.
In particular, we obtain the estimate \eqref{308} for $\Omega_{t_0,t_1}$ with the same $N'$ and the constant $N''$ bounded by $\mu(t_0,t_1)$.
Since $ \mu(t_{0}, t_1) \le \frac{1}{2} $, it follows that
\begin{equation} 
	\label{309}
\|u\|_{\mathfrak{H}(\Omega_{t_0, t_1})}  \le N \left( \|g\|_{L_{p,q}(\Omega_{t_0, t_1})}+  \|f\|_{L_{p_1,q_1}(\Omega_{t_0, t_1})} \right), 
\end{equation}
where $N=N(d, \delta, p, q, p_1, q_1,\ell_1, r_1, \ell_2, r_2, \ell_3, r_3, K, C, T)$. 
In particular, $N$ is independent of $t_1$.


Define 
$$ 
\bar{u}(t,x) =  \left\{
\begin{array}{ccc} 
u(t,x) & \textrm{ if } & t_0 <t <t_1, \\
u(2t_1 -t, x ) & \textrm{ if } & t_1 \le t <t_2,
\end{array} \right.
$$
and set
$v = u - \bar{u}$ on $\Omega_{t_1,t_2}$.
Also set $\bar{a}^{ij}(t,x)$, $\bar{a}^{i}(t,x)$, $\bar{b}^i(t,x)$, $\bar{c}(t,x)$, $\bar{g}_i(t, x)$, and $\bar{f}(t, x)$ to be the odd extensions of $a^{ij}(t,x)$, $a^i(t,x)$, $b^i(t,x)$, $c(t,x)$ $g_i(t, x)$, and $f(t, x)$ with respect to $t$,  respectively. More precisely, for instance,
$$ 
\bar{a}^{ij}(t,x) =  \left\{
\begin{array}{ccc} 
a^{ij}(t,x) & \textrm{ if } & t_0 <t <t_1, \\
-a^{ij}(2t_1 -t, x ) & \textrm{ if } & t_1 \le t <t_2. \\
\end{array} \right.
$$
Then $v$ satisfies
\begin{multline*}
  -v_t + D_i \left( a^{ij} D_j v +a^{i} v \right) + b^i D_i v +cv   =D_i( g_i - \bar{g_i})  + f - \bar{f} \\
+ D_i \left( (\bar{a}^{ij}-a^{ij}) D_j \bar{u} + (\bar{a}^{i}-a^i) \bar{u} \right) + (\bar{b}^i-b^i) D_i \bar{u} + (\bar{c}-c) \bar{u} 
\end{multline*}
in $\Omega_{t_1,t_2}$ with $v(t_1,\cdot) = 0$ on $\Omega$ as in \eqref{eq0310_01} with $t_1$ in place of $0$.

Since $\mu(t_1, t_2) \le \frac{1}{2}$, by applying the above argument for the estimate \eqref{309} to $v$ and $\Omega_{t_1, t_2}$ in place of $u$ and $\Omega_{t_0, t_1}$, respectively, it follows that
\begin{align*}
&\|v\|_{\mathfrak{H}(\Omega_{t_1, t_2})} \le N \|g\|_{L_{p,q}(\Omega_{t_0, t_2})}+N \|f\|_{L_{p_1,q_1}(\Omega_{t_0, t_2})}
\\
&+ N \| Du \|_{L_{p,q}(\Omega_{t_0,t_1})} 
+ N \| (\bar{a}^{i}-a^i )  \bar{u} \|_{L_{p, q}(\Omega_{t_1, t_2})}
\\
&+ N \| (\bar{b}^i-b^i) D_i \bar{u} \|_{L_{p_2,q_2}(\Omega_{t_1, t_2})} +  N \| (\bar{c}-c) \bar{u} \|_{L_{p_3,q_3}(\Omega_{t_1, t_2})}.
\end{align*}
We now apply the inequalities in (i)--(iii) and the estimate \eqref{309} to the last four terms to get
$$
\|v\|_{\mathfrak{H}(\Omega_{t_1, t_2})} \le N \left( \|g\|_{L_{p,q}(\Omega_{t_0 ,t_2})}+ \|f\|_{L_{p_1,q_1}(\Omega_{t_0, t_2})} \right),
$$
where $N = N(d, \delta, p, q, p_1, q_1, \ell_1, r_1, \ell_2, \ell_3, r_2, r_3, \|a^i\|, \|b^i\|, \|c\|, K, C, T)$.

Using $u= v+ \bar{u}$ and by \eqref{309}, we get
$$
\|u\|_{\mathfrak{H}(\Omega_{t_0, t_2})} \le N  \left( \|g\|_{L_{p,q}(\Omega_{t_0 ,t_2})}+N \|f\|_{L_{p_1,q_1}(\Omega_{t_0, t_2})} \right).
$$

Now, thanks to $ \mu(t_{k-1}, t_k) \le \frac{1}{2} $, we repeatedly apply the above argument on the time intervals $[t_2, t_3]$, $[t_3, t_4]$, $\ldots$, $[t_{n_1-1}, t_{n_1}]$ 
to derive
\begin{equation}
							\label{eq0604_01}
\|u\|_{\mathfrak{H}(\Omega_{t_0, t_{n_1}})} \le N \|g\|_{L_{p,q}(\Omega_{t_0 ,t_{n_1}})}+N \|f\|_{L_{p_1,q_1}(\Omega_{t_0, t_{n_1}})},
\end{equation}
where
$$
N=N(d, \delta, p, q, p_1, q_1, \ell_1, r_1, \ell_2, \ell_3, r_2, r_3, \|a^i\|, \|b^i\|, \|c\|, K, C, n_1, T)
$$
$$
= N(d,\delta, p, q, p_1, q_1, \ell_1, r_1, \ell_2, \ell_3, r_2, r_3, a^i, b^i, c, K, C, T).
$$
We there conclude the estimate \eqref{403} upon noting that $(0, T) \times \Omega = \Omega_T = \Omega_{t_0, t_{n_1}} $.

To complete the proof, we prove the unique solvability of the equation with lower-order terms.
Thanks to the a prior estimate proved above, we only prove the existence of a  solution.
Denote
$$
\mathcal{P}_\lambda u = -u_t +D_i \left( a^{ij}(t,x) D_j u + \lambda a^i(t,x) u \right)  + \lambda b^i(t,x) D_i u  + \lambda c(t,x) u,
$$
where
$\lambda \in [0,1]$.
Then from the proof above for the a priori estimate, we see that the estimate \eqref{403} holds independent of $\lambda \in [0,1]$.
Moreover, we have a solution $u \in \mathring{W}_{p,q}^{1,0}(\Omega_T) \cap \mathfrak{H}(\Omega_T)$ of the equation $\mathcal{P}_0 u = D_i g_i + f$ in $\Omega_T$ with $u(0,\cdot) = 0$ on $\Omega$.
Then using the method of continuity argument, we find a solution $u \in \mathring{W}_{p,q}^{1,0}(\Omega_T) \cap \mathfrak{H}(\Omega_T)$ of the equation \eqref{402} with $u(0,\cdot)=0$ on $\Omega$.
In particular, when proceeding in this way, we utilize the fact that the solutions to $\mathcal{P}_\lambda u = D_i g_i + f$, $\lambda \in [0,1]$, belong to $\mathfrak{H}(\Omega_T)$.

The theorem is proved.
\end{proof}
  
\begin{proof}[Proof of Theorem \ref{conormal}]
We follow the proof of Theorem \ref{theo401}. First note that Lemma \ref{190405@lem1} also holds for conormal derivative problems with $u \in W_{p,q}^{1,0}(\Omega)$ in place of $u \in \mathring{W}_{p,q}^{1,0}(\Omega)$. 
Indeed, the duality argument adopted in Lemma \ref{190405@lem1} is available for the conormal derivative problem.
Then the remaining is the same as in the proof of Theorem \ref{theo401}.
 \end{proof}

\section{Solutions in $V_{p,q}$ for $p \geq 2$ and $q \ge 2$}
\label{Vspace}

In this section we give further regularity results for the solution of the equation \eqref{402} with the Dirichlet boundary condition case or the conormal derivative boundary condition case.

We first provide some well-known embedding inequalities for the space $\mathring{V}_{p,q}$, the elements of which have the vanishing lateral boundary condition.
Note that Lemma \ref{multi} (and its corollary) is irrelevant to the regularity of $\partial\Omega$ because the zero lateral boundary condition is considered.

Let $1 < p, q < \infty$. A solution $u \in \mathring{V}_{p,q}(\Omega_T)$ of \eqref{402} with $u(0, \cdot) = 0$ on $\Omega$ is understood in the sense of the integral identity
\begin{multline}
\label{eq0415_01}
\int_{0}^\tau \int_\Omega \left( u \varphi_t  -a^{ij} D_ju D_i \varphi -a^i u D_i\varphi +b^i D_iu \, \varphi + c u \varphi   \right) \, dx\,dt \\
= \int_{\Omega} u (\tau, x) \varphi(\tau,x) \, dx +  \int_{0}^\tau \int_\Omega \left( - g_i D_i \varphi + f\varphi \right) \, dx \, dt  
\end{multline}
for almost all $\tau \in (0, T)$ and smooth test functions $\varphi$ defined in the closure of $(0, \tau) \times \Omega$ and
vanishing in a neighborhood of the lateral boundary of $(0, \tau) \times \Omega$. 
We note that if $u \in \mathring{V}_{p, q}(\Omega_T)$ is a solution to \eqref{402} with $u(0, \cdot) = 0$ on $\Omega$, 
then $v(t, x) := u(t,x)\chi_{t \ge 0}$ is in $\mathring{V}_{p, q}((-\infty, T) \times \Omega)$ and $\|v\|_{V_{p,q}((-\infty, T)\times \Omega)}$ is equal to $\|u\|_{V_{p,q}(\Omega_T)}$.
Also note that if $u \in \mathring{V}_{p,q}(\Omega_T)$, the two notions \eqref{eq0415_01} and \eqref{eq0310_01} of solutions become equivalent.

\subsection{The Dirichlet case}

\begin{lemma} 
\label{multi}
Let $u\in \mathring{V}_{p,q}((S, T) \times \Omega)$, where $-\infty\le S<T\le \infty$ and $p,q\in [1,\infty)$.
Then for any $\kappa \in[\max\{p,d\}, \infty]$ satisfying $\kappa >p$ if $p=d \ge2$, we have 
$$
\|u\|_{L_{\frac{\kappa p}{\kappa -p},\frac{\kappa q}{d}}((S, T) \times \Omega)} 
\le N \left(  \operatorname*{ess\,sup}_{S<t<T} \| u(t, \cdot) \|_{L_p(\Omega)} + \Vert Du \Vert_{L_{p,q}\left( (S, T) \times \Omega \right)} \right),
$$
where $N=N(d,p,\kappa)>0$.
Here, $\frac{\kappa p}{\kappa -p}=p$ if $\kappa=\infty$ and $\frac{\kappa p}{\kappa -p}=\infty$ if $\kappa=p$.

\end{lemma}

\begin{proof}
The case with $\kappa=\infty$ follows from the definition of $\mathring{V}_{p,q}((S, T) \times \Omega)$.
Assume that $\kappa<\infty$.
Considering the zero extension of $u$ outside $\Omega$ and by the interpolation inequality for functions in $\mr^d$ (see, for instance, \cite[Theorem 12.83]{MR3726909}), we have 
$$
\|u(t, \cdot)\|_{L_{\frac{\kappa p}{\kappa -p}}(\Omega)}\le N \|u(t, \cdot)\|_{L_{p}(\Omega)}^{1-\frac{d}{\kappa}}\|Du(t, \cdot)\|_{L_p(\Omega)}^{\frac{d}{\kappa}} 
$$
for almost every $t\in (S, T)$, where $N=N(d,p,\kappa)>0$.
Taking $L_{\frac{\kappa q}{d}}$ norm in $(S, T)$ for both sides, we obtain 
$$
\|u\|_{L_{\frac{\kappa p}{\kappa -p},\frac{\kappa q}{d}}((S, T) \times \Omega)} 
\leq  N \left( \operatorname*{ess\,sup}_{S<t<T} \| u(t, \cdot) \|_{L_p(\Omega)} \right)^{1- \frac{d}{\kappa}}   \|Du\|_{L_{p,q}((S, T) \times \Omega)}^{\frac{d}{\kappa}}. 
$$
We then get the desired estimate from Young's inequality.
The lemma is proved.
\end{proof}

\begin{corollary}
\label{cor0302_1}

Let $u\in \mathring{V}_{p,q}(\Omega_T)$, where $p,q\in [2,\infty)$.
Then we have 
\begin{equation*}
\| u \|_{L_{\frac{(p-1)p_1}{p_1-1}, \frac{(p-1)q_1}{q_1-1}}(\Omega_T)}  \leq N_*T^{N_{**}}  \| u \|_{V_{p,q}(\Omega_T) },
\end{equation*}
where $N_*=N_*(d,p,p_1)>0$ and $N_{**}=N_{**}(d,p,q,p_1,q_1)\ge 0$, provided that the pair $(p_1,q_1)$ satisfies \eqref{40111} and \eqref{401}.
Here, $\frac{(p-1)p_1}{p_1-1}=\infty$ if $p_1=1$  and $\frac{(p-1)q_1}{q_1-1}=\infty$ if $q_1=1$.
\end{corollary}

\begin{proof}
Note that if $q_1=1$, by the fact that $q \geq 2$, that is, $d/p_1+2 \leq 2 + d/p$ and $p_1 \in [1,p]$, it holds that  
$$
p=p_1=\frac{(p-1)p_1}{p_1-1} \quad \text{and} \quad q=2.
$$
Thus, the case with $q_1=1$ follows from the definition of the space $\mathring{V}_{p,q}(\Omega_T)$. 

Assume that $q_1>1$, and set
$$
p_0 = \frac{(p-1)p_1}{p_1-1},   \quad q_0 = \frac{(p-1)q_1}{q_1-1},  ~ \  \textrm{ and } \ ~ \kappa= \frac{p_0 p}{p_0 -p}=\frac{pp_1(p-1)}{p-p_1}.
$$
Then we have $q_0\le \frac{\kappa q}{d}$.
Indeed, by \eqref{401} and $q\ge 2$, we get 
$$
 \frac{1}{q_0} - \frac{d}{\kappa q} = \frac{1}{p-1} \left( 1- \frac{1}{q_1} - \frac{1}{q} \left( \frac{d}{p_1} - \frac{d}{p} \right) \right)
\geq \frac{1}{p-1} \left(  1- \frac{1}{q_1} -\frac{1}{q} - \frac{2}{q^2} + \frac{2}{q q_1} \right) 
$$
$$
\geq \frac{1}{(p-1)q} \left(  q- \frac{2}{q} -1 - \frac{q-2}{q_1} \right) 
\geq \frac{1}{(p-1)q} \left(  q-2 - \frac{q-2}{q_1}   \right)  \geq 0. 
$$
Thus by H\"older's inequality and $p_0=\frac{\kappa p}{\kappa -p}$, we obtain 
\begin{equation*}	
\| u \|_{L_{p_0, q_0}(\Omega_T) }  \leq
T^{ \frac{1}{q_0} - \frac{d}{\kappa q}}   \| u \|_{L_{\frac{\kappa p}{\kappa -p},\frac{\kappa q}{d}}(\Omega_T) }.
\end{equation*}
Thanks to Lemma \ref{multi}, it is sufficient to show that $\kappa$ satisfies the conditions in the hypothesis of the lemma. 
Obviously, we have $\kappa \ge p$.
By \eqref{401} and $p\ge2$, it holds that
$$
\frac{1}{\kappa}=\frac{1}{p-1}\left(\frac{1}{p_1}-\frac{1}{p}\right)\le  \frac{1}{d}+\frac{2}{qd}-\frac{2}{q_1d}\le \frac{1}{d},
$$
which gives $\kappa \ge d$.
In particular, if $p=d>2$, then $\kappa>p$ because the first $``\le"$ in the above inequalities can be replaced by $``<"$.
If $p=d=2$, then by using the fact that $p_1>1$ (see \eqref{40111}), we have 
$$
\frac{1}{\kappa}=\frac{1}{p_1}-\frac{1}{2}<\frac{1}{2},
$$
which implies $\kappa>2$.
\end{proof}

In the proposition below we show that the solution $u$ from Theorem \ref{theo401} are in fact in $V_{p,q}(\Omega_T)$ if $p,q \in [2,\infty)$.
In the proof, we only use the  uniform ellipticity of the coefficients $a^{ij}$ once we know the estimate of $u$ in term of the $\mathfrak{H}(\Omega_T)$ norm.

\begin{proposition}		\label{190405@lem2}
Let $p, q \in [2, \infty)$. Then under the same assumptions as in Theorem \ref{theo401}, there exists a unique $u\in \mathring{V}_{p, q}(\Omega_T)$ satisfying \eqref{402} with $u(0, \cdot) = 0$ on $\Omega$.
Moreover, 
\begin{equation}		\label{190412@eq1}
\operatorname*{ess\, sup}_{ 0 <t<T}\|u(t, \cdot)\|_{L_{p}(\Omega)} 
\le  N \left(   \| g\|_{L_{p,q} (\Omega_{T})} +  \|f\|_{L_{p_1, q_1}(\Omega_{T})} \right),
\end{equation}
where $N=N(d, \delta, p, q, p_1, q_1, \ell_1, r_1, \ell_2, \ell_3, r_2, r_3, a^i, b^i, c, K, C, T)$.
\end{proposition}

\begin{proof} By Theorem \ref{theo401}, we have a unique $u \in \mathring{W}_{p,q}^{1,0}(\Omega_T)$ satisfying \eqref{402} with $u(0, \cdot) = 0$ on $\Omega$.
We also know from the proof that $u \in \mathfrak{H}(\Omega_T)$.
Let $\tilde{g}_i = g_i - a^iu$, $f_1 = f$, $f_2 = -b^iD_iu$ and $f_3 = -cu$. We first extend the equation \eqref{402} from $\Omega_{T}$ to $(-1,T) \times \Omega$. 
Indeed, we extend $u$, $\tilde{g}_i$, $f_k$ to be zero for $t \le 0$ and $a^{ij} = \delta^{ij}$ for $t \le 0$. 
For the sake of simplicity, we still denote by $u$, $\tilde{g}_i$, $f_k$, $a^{ij}$ the extended ones.
Then one can check that $u \in \mathring{W}_{p, q}^{1, 0}((-1, T)\times \Omega)$ satisfies $u(-1, \cdot) = 0$ on $\Omega$ and
\begin{equation*}
-u_t + D_i(a^{ij}D_ju) = D_i\tilde{g}_i + \sum_{k=1}^{3}f_k \quad \textrm{in} \quad (-1, T) \times \Omega
\end{equation*}
in the sense of \eqref{eq0310_01}.

Denote by $v_h$ the Steklov average of a function $v$ with step size $h  \in (0,1)$, that  is, 
$$
v_h(t,x)=\frac{1}{h}\int_{t}^{t+h} v(s,x)\,ds.
$$
Then by following the argument in \cite[pp. 141--143]{MR0241822}, we have that
\begin{equation}		\label{190401@A1}
\begin{gathered}
\frac{p-1}{p}\int_\Omega |u_h(\tau, \cdot)|^p\,dx-\frac{p-1}{p}\int_\Omega |u_h(-1,\cdot)|^p\,dx\\
+(p-1)\int_{  (-1, \tau) \times \Omega  }(a^{ij}D_ju)_h D_i u_h |u_h|^{p-2} \, dx\, dt \\
=(p-1)\int_{ (-1, \tau) \times \Omega } (g_i)_h D_i u_h |u_h|^{p-2}\,dx\,dt-\sum_{k=1}^3 \int_{ (-1, \tau) \times \Omega } (f_k)_h u_h |u_h|^{p-2} \,dx \, dt
\end{gathered}
\end{equation}
for $\tau \in (-1,T)$ and $h \in (0, T-\tau)$.
Observe that 
$$
|\tilde{g}_h| |D u_h| |u_h|^{p-2}\le \frac{\delta }{2} |Du_h|^2|u_h|^{p-2}+N |\tilde{g}_h|^2 |u_h|^{p-2},
$$
and that 
$$
\begin{aligned}
&(a^{ij}D_ju)_h D_i u_h |u_h|^{p-2}\\
&=a^{ij} D_j u_h D_i u_h |u_h|^{p-2}+\left((a^{ij}D_j u)_h -a^{ij} D_j u_h\right)D_i u_h |u_h|^{p-2}\\
&\ge \delta |Du_h|^2|u_h|^{p-2}+\left((a^{ij}D_j u)_h -a^{ij} D_j u_h\right)D_i u_h |u_h|^{p-2}.
\end{aligned}
$$
Thus, from \eqref{190401@A1} with the fact that $\|u_h(-1, \cdot )\|_{L_{p}(\Omega)} = 0$ we get
\begin{equation}		\label{190401@A2}
\|u_h(\tau, \cdot)\|_{L_{p}(\Omega)}^p \le N(I_h^0+I_h^1+I_h^2)
\end{equation}
for any $\tau\in (-1,T)$ and $h \in (0, T-\tau)$,
where $N=N(d,\delta,p)$ and
\begin{align*}
I_h^0 &=\int_{(-1, \tau) \times \Omega}\left|(a^{ij}D_j u)_h -a^{ij} D_j u_h\right| |D_i u_h| |u_h|^{p-2}\,dx \, dt,
\\
I_h^1 &=\int_{(-1,\tau) \times \Omega} |\tilde{g}_h|^2|u_h|^{p-2}\,dx \, dt,
\\
I_h^2 &= \sum_{k=1}^3 \int_{(-1,\tau) \times \Omega}  | (f_k)_h|   |u_h|^{p-1} \, dx \, dt.
\end{align*}
For $I_h^0$, we have
$$
I_h^0 \le  \int_{-1}^{\tau} \|  (a^{ij}D_ju)_h(t, \cdot) - a^{ij}D_ju_h(t, \cdot)\|_{L_p(\Omega)} \| Du_h(t, \cdot)\|_{L_p(\Omega)}\| u_h(t, \cdot)\|_{L_p( \Omega)}^{p-2} \, dt
$$
$$
\le \left(\operatorname*{ess \, sup}_{-1<t <\tau} \Vert u_h(t, \cdot) \Vert_{L_p(\Omega)}\right)^{p-2}    \|  (a^{ij}D_j u)_h -a^{ij} D_j u_h  \|_{L_{p, q}}  \| Du_h \|_{L_{p, q'}},
$$
$$
\le   (T+1)^{1-\frac{2}{q}} \left(\operatorname*{ess \, sup}_{-1<t <\tau} \Vert u_h(t, \cdot) \Vert_{L_p(\Omega)}\right)^{p-2}    \|  (a^{ij}D_j u)_h -a^{ij} D_j u_h  \|_{L_{p, q}} \| Du_h \|_{L_{p, q}},
$$
where $L_{p,q} = L_{p,q}(\Omega_{-1,\tau})$ and $L_{p,q'} = L_{p,q'}(\Omega_{-1,\tau})$.
For $I_h^1$, we have
\begin{align*}
I_h^1 &\le \int_{-1}^{\tau} \| u_h(t, \cdot) \|_{L_p(\Omega)}^{p-2} \| \tilde{g}_h(t, \cdot) \|_{L_p(\Omega)}^2 dt 
\\
&\le (T+1)^{1 - \frac{2}{q}} \operatorname*{ess \, sup}_{-1<t <\tau} \Vert u_h(t, \cdot) \Vert_{L_p(\Omega)}^{p-2} \| \tilde{g}_h \|_{L_{p, q}(\Omega_{-1, \tau})}^2.
\end{align*}

To estimate $I_h^2$, we use H\"older's inequality and Corollary \ref{cor0302_1} to get
$$
\int_{ \Omega_{-1,\tau} } |(f_k)_h| \, |u_h|^{p-1} \, dx \, dt     \leq \| (f_k)_h \|_{L_{p_k,q_k}( \Omega_{-1,\tau} )} \| u_h \|_{L_{\frac{(p-1)p_k}{p_k - 1}, \frac{(p-1)q_k}{q_k - 1}}( \Omega_{-1,\tau} )}^{p-1}
$$
$$
\le N \|(f_k)_h\|_{L_{p_k,q_k}( \Omega_{-1,\tau} )} \| u_h \|_{V_{p,q}( \Omega_{-1,\tau} ) }^{p-1}
$$
where $N = N(d,p,q,p_k,q_k,T)$ and $k = 1, 2, 3$.

Combining \eqref{190401@A2} with the estimates of $I_h^0$, $I_h^1$, and $I_h^2$, we have
\begin{align}
& \|u_h(\tau, \cdot)\|_{L_{p}(\Omega)}^p \leq N \| u_h \|_{V_{p, q}(\Omega_{-1, \tau})}^{p-2}\| \tilde{g}_h \|_{L_{p, q}(\Omega_{-1, \tau})}^2 \nonumber
\\
& \quad + N \| u_h \|_{V_{p, q}(\Omega_{-1, \tau})}^{p-1} \sum_{k=1}^3 \|(f_k)_h\|_{L_{p_k,q_k}( \Omega_{-1,\tau} )} \nonumber
\\
& \quad + N \| u_h \|_{V_{p, q}(\Omega_{-1, \tau})}^{p-2} \|  (a^{ij}D_j u)_h -a^{ij} D_j u_h  \|_{L_{p, q}(\Omega_{-1, \tau})} \| Du_h \|_{L_{p, q}(\Omega_{-1, \tau})} \nonumber
\\
& \le \frac{1}{2}\| u_h \|_{V_{p, q}(\Omega_{-1, \tau})}^p \nonumber
\\
& \quad + N \left(   \| u_h \|_{W_{p, q}^{1, 0}(\Omega_{-1, \tau})}^p + \| \tilde{g}_h \|_{L_{p, q}(\Omega_{-1, \tau})}^p + \sum_{k=1}^3 \|(f_k)_h\|_{L_{p_k,q_k}( \Omega_{-1,\tau} )}^p   \right) \nonumber
\end{align}
where $N = N(d, \delta, p, q, p_k, q_k, T)$.
Since
$$
\lim_{h \to 0} \| Du_h \|_{L_{p,q} ( \Omega_{-1,\tau} )}^p = \| Du \|_{L_{p,q} ( \Omega_{-1,\tau} )}^p \le \| Du \|_{L_{p,q} ( \Omega_{T} )}^p,
$$
by letting $h \to 0$ and using the estimate (i)--(iii) in the proof of Theorem \ref{theo401} as well as the estimate \eqref{eq0604_01}, we arrive at \eqref{190412@eq1}. 
\end{proof}

\subsection{The conormal case}
As in the Dirichlet case, we first provide some embedding inequalities for the space $V_{p,q}$.
Note that in this case we assume that $\Omega$ is a John domain, the definition of which is given below.

Note that solutions $u \in V_{p,q}(\Omega_T)$ of \eqref{602}, with $u(0, \cdot) = 0$ on $\Omega$, 
are understood in an analogous way to the Dirichlet boundary value problem.
More precisely, we define solutions to the conormal boundary value problem using \eqref{eq0415_01} with test functions having (not necessarily) non-zero values on the lateral boundary of the parabolic domain.
As in the Dirichlet case, if $u \in V_{p,q}(\Omega_T)$,
then \eqref{eq0310_01} with $T$ is equivalent to \eqref{eq0415_01} with almost all $\tau \in (0,T)$.

\begin{definition} \label{john}
A bounded domain $\Omega$ in $\mr^d$ is a John domain with center $z_0 \in \Omega$ and constant $M \ge 1$ if for each $z \in \Omega\setminus \{z_0\}$, there is a rectifiable curve $\eta(z, z_0) \subset \Omega$ connecting $z$ and $z_0$ such that
\begin{equation*}
|\eta(z, x)| \le Md(x) \quad \textrm{for}\quad x \in \eta(z, z_0)
\end{equation*}
where $d(x) = \operatorname{dist}(x, \partial\Omega)$ and $|\eta(z, x)|$ denotes the length of the subcurve $\eta(z, x) \subset \eta(z, z_0)$ connecting $z$ and $x$.
\end{definition}

If there is no confusion, hereafter, we simply say ``John domain" instead of ``John domain with center $z_0$ and constant $M$".

\begin{remark}
	\label{john2}
Every Lipschitz domain is a John domain.
It is known that $\Omega$ is a John domain if it is a $(R_0, \gamma)$-Reifenberg flat domain. (See Definition \ref{rfd} for the precise definition.)
In this case, the constants $M$ and $d(z_0)$ may depend on $R_0$, $\gamma$ and $diam(\Omega)$. 	
Indeed, a Reifenberg flat domain is a uniform domain for sufficiently small $\gamma>0$ and consequently it is a John domain. 
We refer the reader to \cite{MR2186550, MR3186805} for more details.
\end{remark}

\begin{lemma} 
\label{multi2}
Let $\Omega$ be a John domain and $u\in V_{p,q}((S, T) \times \Omega)$, where $-\infty\le S<T\le \infty$ and $p,q\in [1,\infty)$.
Then for any $\kappa \in[\max\{p,d\}, \infty]$ satisfying $\kappa >p$ if $p \ge d \ge 2$, we have 
$$
\|u\|_{L_{\frac{\kappa p}{\kappa -p},\frac{\kappa q}{d}}((S, T) \times \Omega)} 
\le N \left(  \operatorname*{ess\,sup}_{S<t<T} \| u(t, \cdot) \|_{L_p(\Omega)} + \Vert u \Vert_{W_{p,q}^{1,0}\left( (S, T) \times \Omega \right)} \right),
$$
where $N=N(d,p,\kappa, |\Omega|, M, d(z_0) )>0$.
Here, $\frac{\kappa p}{\kappa -p}=p$ if $\kappa=\infty$ and $\frac{\kappa p}{\kappa -p}=\infty$ if $\kappa=p$.

\end{lemma}

\begin{proof}
The case with $\kappa=\infty$ follows from the definition of $V_{p,q}((S, T) \times \Omega)$.
Assume that $\kappa<\infty$ and $d\ge2$.
By \cite[Lemma 3.4]{MR4025954}, we have 
$$
\|u(t, \cdot)\|_{L_{\frac{\kappa p}{\kappa -p}}(\Omega)}\le N \|u(t, \cdot)\|_{L_{p}(\Omega)}^{1-\frac{d}{\kappa}}\|u(t, \cdot)\|_{W_p^1(\Omega)}^{\frac{d}{\kappa}},
$$
for almost every $t\in (S, T)$, where $N=N(d,p,\kappa, M, d(z_0), | \Omega |)>0$.
Taking $L_{\frac{\kappa q}{d}}$ norm in $(S, T)$ for both sides, we obtain 
$$
\|u\|_{L_{\frac{\kappa p}{\kappa -p},\frac{\kappa q}{d}}((S, T) \times \Omega)} 
\leq  N \left( \operatorname*{ess\,sup}_{S<t<T} \| u(t, \cdot) \|_{L_p(\Omega)} \right)^{1- \frac{d}{\kappa}}   \|u\|_{W_{p,q}^{1,0}((S, T) \times \Omega)}^{\frac{d}{\kappa}}. 
$$
We then get the desired estimate from Young's inequality.

If $d = 1$, note that $\Omega$ is an interval, say $(0, 1)$. 
Then we extend $u \in V_{p,q}((S, T) \times \Omega) = V_{p,q}((S, T) \times (0, 1))$ to $v \in V_{p,q}((S, T) \times (-1, 2))$ 
by an appropriate extension with multiplying a cut-off function $\eta$ whose support is in $(-1, 2)$ and $\eta = 1$ on $(0, 1)$.
Then, from the case $d=1$ of Lemma \ref{multi}, we easily obtain
\begin{align*}
&\|u\|_{L_{\frac{\kappa p}{\kappa -p},\frac{\kappa q}{d}}((S, T) \times (0, 1))} \le \|v\|_{L_{\frac{\kappa p}{\kappa -p},\frac{\kappa q}{d}}((S, T) \times (-1, 2))}   
\\
&\le N(p, \kappa) \left( \operatorname*{ess\,sup}_{S<t<T} \| v(t, \cdot) \|_{L_p((-1, 2))} +   \|Dv\|_{L_{p,q}((S, T) \times (-1, 2))} \right)  
\\
&\le N \left( \operatorname*{ess\,sup}_{S<t<T} \| u(t, \cdot) \|_{L_p((0, 1))} +   \|u\|_{W_{p,q}^{1,0}((S, T) \times (0, 1))} \right).
\end{align*}
By dilation and translation, we also obtain the same estimate for $\Omega = (a, b)$, 
with the constant $N$ depending on $b-a = |\Omega|$, as in the case $d \ge 2$.
The lemma is proved.
\end{proof}

\begin{corollary}
\label{cor0502}

Let $\Omega$ be a John domain and $u\in V_{p,q}(\Omega_T)$, where $p,q\in [2,\infty)$. Then we have 
\begin{equation*}
\| u \|_{L_{\frac{(p-1)p_1}{p_1-1}, \frac{(p-1)q_1}{q_1-1}}(\Omega_T)}  \leq N_*T^{N_{**}}  \| u \|_{V_{p,q}(\Omega_T) },
\end{equation*}
where $N_*=N_*(d,p,p_1, |\Omega|, M, d(z_0))>0$ and $N_{**}=N_{**}(d,p,q,p_1,q_1)\ge 0$, provided that the pair $(p_1,q_1)$ satisfies \eqref{40111} and \eqref{401}. 
Here, $\frac{(p-1)p_1}{p_1-1}=\infty$ if $p_1=1$  and $\frac{(p-1)q_1}{q_1-1}=\infty$ if $q_1=1$.

\end{corollary}

\begin{proof}
By the same reason in the proof of Corollary \ref{cor0302_1}, we obtain 
\begin{equation*}		
\| u \|_{L_{p_0, q_0}(\Omega_T) }  \leq
T^{ \frac{1}{q_0} - \frac{d}{\kappa q}}   \| u \|_{L_{\frac{\kappa p}{\kappa -p},\frac{\kappa q}{d}}(\Omega_T) }.
\end{equation*}
Thanks to Lemma \ref{multi2}, it is sufficient to show that $\kappa = \frac{pp_1(p-1)}{p-p_1}$ satisfies the hypothesis of Lemma \ref{multi2}. 
Obviously, we have $\kappa \ge p$ and $\kappa \ge d$ by the relation \eqref{401} and the fact that $p\ge2$. Also if $d \ge 2$, since \eqref{40111} implies $p_1 > 1$,  we have $\kappa > p$.
\end{proof}

\begin{proposition}		\label{lem0603_1}

Let $p, q \in [2, \infty)$ and $\Omega$ be a John domain. Then under the same assumptions as in Theorem \ref{conormal}, there exists a unique $u\in {V}_{p, q}(\Omega_T)$ satisfying \eqref{602} with $u(0, \cdot) = 0$ on $\Omega$.
Moreover, 
\begin{equation*}
\operatorname*{ess\, sup}_{ 0 <t<T}\|u(t, \cdot)\|_{L_{p}(\Omega)} 
\le  N \left(   \| g\|_{L_{p,q} (\Omega_{T})} +  \|f\|_{L_{p_1, q_1}(\Omega_{T})} \right),
\end{equation*}
where $N=N(d, \delta, p, q, p_1, q_1,\ell_1, r_1, \ell_2, \ell_3, r_2, r_3, a^i, b^i, c, K, C, T, |\Omega|, d(z_0), M)$.
\end{proposition}

\begin{proof}
The proof is almost the same as that of Proposition \ref{190405@lem2}. Indeed, we use Corollary \ref{cor0502} instead of Corollary \ref{cor0302_1} to estimate $I_h^2$ apearing in the proof of Proposition \ref{190405@lem2}.
\end{proof}

\section{Appendix}
\label{appendix}

In this section, we illustrate some sufficient conditions on $\Omega$ (or on $\partial \Omega$) so that the domain satisfies Assumption \ref{dom} ($\Omega$).
We also give brief remarks about the parameters on which the constants depend in the main theorems as well as the elliptic case with unbounded coefficients.
Throughout the section, we use the following notation.
$$
\Omega(x_0,\rho) = \Omega \cap B_{\rho}(x_0) \quad \textrm{and} \quad  \Omega(\rho) = \Omega \cap B_{\rho}(0).
$$

\begin{definition}[$R_0, \gamma$] \label{rfd}
We say $\Omega \subset \mr^d$ is $(R_0, \gamma)$-Reifenberg flat if there exists a positive constant $R_0$ such that the following holds: for any $x_0 \in \Omega$ and $R \in (0, R_0]$, there is a coordinate system depending on $x_0$ and $R \in (0, R_0]$ such that in the new coordinate system, we have
\begin{equation}
	\label{rfd2}
\{y : x_{01} + \gamma R < y_1\} \cap B_R(x_0) \subset \Omega(x_0, R) \subset \{y : x_{01} - \gamma R < y_1\} \cap B_R(x_0),
\end{equation}
where $x_{01}$ is the first coordinate of $x_0$ in the new coordinate system.
\end{definition} 

We remark that this definition is meaningful when the positive number $\gamma$ is sufficiently small, such as $\gamma \in [0, 1/48]$.

\begin{theorem}[Embedding in $(0, T) \times \Omega$]
	\label{thm602}
Let $p, q \in [1, \infty]$ and $\Omega$ be a bounded $(R_0, \gamma)$-Reifenberg flat domain in $\mr^d$ with $\gamma \le 1/48$. Also let $p_0 \in [p,\infty]$, $q_0 \in [q,\infty]$, $p_k \in (0,\infty]$, $q_k \in (0,\infty]$, $k=1,\ldots,m$, satisfy either the condition (i) or (ii) in Assumption \ref{dom}.
Then for $u \in W_{p, q}^{1, 0}(\Omega_T)$, if $u_t = D_ig_i + \sum_{k=1}^m f_k$ where $g = (g_1, \ldots, g_d) \in \left( L_{p, q}(\Omega_T) \right)^d$ and $f_k \in L_{p_k, q_k}(\Omega_T)$, $k = 1, \ldots, m$, we have
\begin{equation} \label{imb2}
\|u\|_{L_{p_0,q_0}(\Omega_T)} \le N\left( \|u\|_{W_{p,q}^{1, 0}(\Omega_T)} + \|g\|_{L_{p, q}(\Omega_T)} + \sum_{k=1}^m \| f_k \|_{L_{p_k, q_k}(\Omega_T)} \right)
\end{equation}
where $N = N(d, p, q, p_0, q_0, p_k, q_k, m, R_0, \operatorname{diam}(\Omega), T)$. 
\end{theorem}

For the proof of Theorem \ref{thm602}, we need the following boundary and interior estimates.

\begin{lemma}[Boundary estimates]
							\label{lem0805_1}
Let $p, q \in [1,\infty]$ and $\Omega \subset \mr^d$ be a $(R_0, 1/48)$-Reifenberg flat domain, $x_0 \in \partial\Omega$, and $R \in (0, R_0/4]$.
Also let $p_0 \in [p,\infty]$, $q_0 \in (q,\infty]$, $p_k \in (0,\infty]$, $q_k \in (0,\infty]$, $k=1,\ldots,m$, satisfy the condition (ii) in Assumption \ref{dom}.
Then for $u \in W_{p, q}^{1, 0}((0, T) \times \Omega(x_0, 2R))$, if $u_t = D_ig_i + \sum_{k=1}^m f_k$ where $g = (g_1, \ldots, g_d) \in \left( L_{p, q}((0, T) \times \Omega(x_0, 2R)) \right)^d$ and $f_k \in L_{p_k, q_k}((0, T) \times \Omega(x_0, 2R))$, $k = 1, \ldots, m$, we have
\begin{multline*}
\|u\|_{L_{p_0,q_0}((0, T) \times \Omega(x_0, R))} \leq N \|u\|_{W_{p,q}^{1, 0}((0, T) \times \Omega(x_0, 2R))} 
\\ +  N \|g\|_{L_{p, q}((0, T) \times \Omega(x_0, 2R))} +  N\sum_{k=1}^m \| f_k \|_{L_{p_k, q_k}((0, T) \times \Omega(x_0, 2R))},
\end{multline*}
where $N = N(d, p, q, p_0, q_0, p_k, q_k, m, R, T)$.
\end{lemma}

\begin{proof} 
By denseness and an extension with respect to $t$, it suffices to prove the desired inequality with $\mr$ in place of $(0,T)$ for $u \in C^\infty(\mr \times \Omega(x_0, 2R))$.
Take $h = 4 \cdot 48 \cdot 24$.
Without loss of generality, we assume that $x_0 = 0$. Set
$$
z_0 := (R/2, 0, \ldots , 0) \in \mr^d
$$
in the coordinate system associated with the origin and $4R$ satisfying \eqref{rfd2}. 
Note that there exist $g_i \in L_{p,q}\left(\mr \times \Omega(2R)\right)$, $i=1,\ldots,d$ and $f_k \in L_{p_k, q_k}(\mr \times \Omega(2R))$, $k = 1, \ldots, m$, such that  $u_t = D_ig_i + \sum_{k=1}^mf_k$ in $\mr \times \Omega(2R)$ in the distribution sense.

For $x \in \Omega(R)$, take a path $\eta(\lambda) = \eta(\lambda; x)$ from $x$ to $z_0$.
Note that there exists a curve on $[0,1]$ satisfying
\begin{equation}
	\label{path_1}
\eta(\lambda) \subset \Omega(7R/4), \quad \operatorname{Lip}(\eta) \le 5R,\quad \operatorname{dist}(\eta(\lambda), \partial\Omega) > \lambda R/h \quad \text{for} \quad \lambda \in [0,1],
\end{equation}
where $\operatorname{Lip}(\eta)$ denotes the Lipschitz constant of $\eta$. For details about \eqref{path_1} we refer the reader to \cite[Appendix]{MR3809039}.
Then, for $(s, z) \in \mr \times B_{R/h}(z_0)$, we define another path from $(t, x) \in \mr \times \Omega(2R)$ to $(s, z)$ by
$$
\tau(\lambda) = \left( (1-\lambda^2)t + \lambda^2s, \eta(\lambda;x) + \lambda(z - z_0) \right) \in \mr \times \Omega(2R), \quad \lambda \in [0, 1].
$$

Take $\phi \in C_0^\infty(\mr^{d+1})$ satisfying
$$
\phi(s,y) = \zeta(s-t) \varphi(y),
$$
where $\zeta(s) \in C_0^\infty(\mr)$, $\varphi(y) \in C_0^\infty(\mr^d)$, and
$$
0 \leq \zeta(s) \leq 1, \quad \zeta = 1 \,\, \text{on} \,\, [(R/2h)^2, 3(R/2h)^2], \quad \operatorname{supp} \zeta \subset \left(0, (R/h)^2\right),
$$
$$
0 \leq \varphi(y) \leq 1, \quad \varphi = 1 \,\, \text{on} \,\, B_{R/2h}, \quad \operatorname{supp} \varphi \subset B_{R/h}.
$$
We further assume that $\varphi$ is radially symmetry so that $\varphi(y) = \varphi(|y|)$.
Then $0 \leq \phi \leq 1$, $|D_x\phi| \leq N/R$, and
$$
\phi \equiv 1 \quad \text{on} \quad [t+(R/2h)^2, t+3(R/2h)^2] \times B_{R/2h}(0),
$$
$$
\operatorname{supp} \phi \subset \left(t, t+(R/h)^2\right) \times B_{R/h}(0).
$$
Let
\begin{align*}
\bar{u} :=& \bar{u}(t) = \frac{1}{\|\phi\|_{L_1(\mr^{d+1})}}\int_{\mr}\int_{\mr^d} u(s,z)\eta(s-t)\varphi(z- z_0) \, dz\,ds
\\
=& \frac{1}{\|\phi\|_{L_1(\mr^{d+1})}}\int_\mr\int_{B_{R/h}(z_0)} u(s,z)\phi(s, z- z_0) \, dz\,ds.
\end{align*}
Then
\begin{equation}
	\label{mean}
\begin{aligned}
 u(t,x) - \bar{u} &= \frac{1}{\|\phi\|}\int_\mr \int_{\mr^d} \left[ u(\tau(0)) - u(\tau(1)) \right] \phi(s, z - z_0)\,dz\,ds 
\\
&= -\frac{1}{\|\phi\|} \int_\mr \int_{\mr^d}  \left[ \int_0^1 \tau'(\lambda)
\cdot 
(\nabla_{t,x}u)\left(\tau(\lambda)\right) \, d\lambda \right] 
\phi(s, z - z_0) \,dz\,ds
\\
&= -\frac{1}{\|\phi\|} (A + B), 
\end{aligned}
\end{equation}
where
$$
A := \int_\mr \int_{\mr^d} \int_0^1 \, 2\lambda(s-t) u_t\left(\tau(\lambda)\right) \phi(s, z-z_0) \, d\lambda \,dz \,ds,
$$
$$
B := \int_\mr \int_{\mr^d} \int_0^1 \left[ \frac{\partial}{\partial\lambda}\left(\eta(\lambda;x) + \lambda(z - z_0) \right)\right] \cdot  
(\nabla_{x}u)\left(\tau(\lambda)\right)
\phi(s, z-z_0) \, d\lambda \, dz\,ds,
$$
\begin{equation}
	\label{phi}
\| \phi \| = \| \phi \|_{L_1(\mr^{d+1})} \ge N(d) R^{d+2}.
\end{equation}

For each $\lambda \in (0,1)$ by the change of variables $(s,z) \to (l,y)$ with 
\begin{equation}
							\label{eq0531_02}
l = (1-\lambda^2) t + \lambda^2 s, \quad y = \eta(\lambda; x) + \lambda(z-z_0),
\end{equation}
we have
\begin{align*}
A &= \int_0^1 \int_\mr \int_{B_{R/h}(z_0)} \, 2\lambda(s-t) u_t(\tau(\lambda)) \phi(s, z-z_0) \, dz \,ds \, d\lambda
\\
&= \int_0^1 \lambda^{-d-2} F(\lambda) \, d\lambda,
\end{align*}
where
$$
F(\lambda) := F(\lambda; t,x)
$$
$$
= \int_\mr \int_{B_{\lambda R/h}(\eta(\lambda))} \frac{2(l-t)}{\lambda} u_t(l, y) \phi \left( \frac{l - (1 - \lambda^2)t}{\lambda^2}, \frac{y - \eta(\lambda ; x)}{\lambda} \right) \, dy \, dl.
$$
Because of the choice of $\phi$ and the properties of $\eta(\lambda;x)$ for $x \in \Omega(R)$ in \eqref{path_1},
we see that
$$
2\frac{(l-t)}{\lambda} \phi \left( \frac{l-(1-\lambda^2)t}{\lambda^2}, \frac{y - \eta(\lambda ; x)}{\lambda} \right) =: \xi(l, y) \in C_0^{\infty}\left(\mr \times \Omega(2R)\right).
$$
Then, since $B_{\lambda R/h}(\eta(\lambda)) \subset \subset \Omega(2R)$ and $u_t = D_ig_i + \sum_{k=1}^m f_k$ in $\mr \times \Omega(2R)$, we have
$$
F(\lambda, t,x) = \int_{\mr}\int_{\Omega(2R)} u_t(l,y)\xi(l,y) dy \, dl = -\int_{\mr} \int_{\Omega(2R)} u(l,y) \xi_t(l,y) dy\,dl
$$
$$
=  \sum_{k=1}^m\int_{\mr}\int_{\Omega(2R)} f_k(l, y)\xi(l, y) dy\,dl - \int_{\mr}\int_{\Omega(2R)} g_i(l,y) D_i\xi(l, y) \, dy\,dl.
$$
Hence, 
\begin{align*}
A =& \sum_{k=1}^m \int_0^1\lambda^{-d-2}\int_{\mr}\int_{\Omega(2R)} f_k(l, y)\xi(l, y) dy\,dl \, d\lambda 
\\
&- \int_0^1\lambda^{-d-2}\int_{\mr}\int_{\Omega(2R)} g_i(l,y) D_i\xi(l, y) \, dy\,dl \, d\lambda =: \sum_{k=1}^mA_{1, k} + A_2.
\end{align*}

We see that
\begin{equation}
							\label{eq0523_02}
0 \leq \frac{l - t}{\lambda} \leq \lambda (R/h)^2 \leq (R/h)^2  , \quad \frac{l-t}{\lambda^2} \leq (R/h)^2
\end{equation}
for $l \in [t, t+ \lambda^2(R/h)^2]$.
Using the fact that $\phi$ has compact support with respect to the time variable in $[t, t+(R/h)^2]$ and \eqref{eq0523_02}, we obtain that, for each $k = 1, \ldots, m$,
\begin{align}
							\label{eq0601_01}
|A_{1, k}| &\leq 2 \int_0^1 \lambda^{-d-2} \int_t^{t + \lambda^2 (R/h)^2} \int_{\Omega(2R)} \left| f_k(l,y)   \frac{l-t}{\lambda}  \phi (\cdot)  \right| \, dy \,dl\, d\lambda
\\
&\leq R^2 \int_0^1 \lambda^{-d-1} \int_t^{t + (R/h)^2} \int_{\Omega(2R)} \left| f_k(l,y)  \phi(\cdot)  \right| \, dy\, dl\, d\lambda, \nonumber
\end{align}
and
\begin{align*}
|A_2| &\leq 2 \int_0^1 \lambda^{-d-2}  \int_{t}^{t + \lambda^2 (R/h)^2} \int_{\Omega(2R)} \left| g_i(l,y)  \frac{l-t}{\lambda^2} (D_i\phi) (\cdot) \right| \, dy\, dl\,  d\lambda
\\
&\leq R^2 \int_0^1 \lambda^{-d-2}  \int_t^{t + (R/h)^2} \int_{\Omega(2R)} \left| g_i(l,y) (D_i\phi)(\cdot)  \right| \, dy \, dl \,  d\lambda,
\end{align*}
where
\begin{equation}
							\label{eq0523_03}
(\phi, D_i \phi)(\cdot) = (\phi, D_i \phi) \left( \frac{l - (1 - \lambda^2)t}{\lambda^2}, \frac{y - \eta(\lambda ; x)}{\lambda} \right).
\end{equation}
By again using the change of variables in \eqref{eq0531_02} and considering the support of $\phi$, we have
$$
|B| \leq \int_t^{t+(R/h)^2}\int_{\Omega(2R)} \int_0^1 \lambda^{-d-2} \left| \dot{\eta}(\lambda;x) + \frac{y - \eta(\lambda ; x)}{\lambda} \right| |\nabla_x u(l,y)| \phi(\cdot) \, d\lambda \, dy \, dl,
$$
where $\phi(\cdot)$ is as in \eqref{eq0523_03}.
Note that 
$$
\phi \left(\frac{l - (1 - \lambda^2)t}{\lambda^2}, \frac{y - \eta(\lambda ; x)}{\lambda}\right) = \zeta\left(\frac{l - t}{\lambda^2}\right) \varphi\left(\frac{y - \eta(\lambda ; x)}{\lambda}\right) = 0
$$
provided that $|x-y| \geq (1/h+5)\lambda R$.
Indeed, by the Lipschitz condition in \eqref{path_1},
\begin{equation*}
|y - \eta(\lambda; x)| \ge |x-y| - |x - \eta(\lambda;x)| \ge \lambda R/h
\end{equation*}
if $|x-y| \ge (1/h + 5)R\lambda$ and $y \in \Omega(2R)$.
On the other hand, if $|x-y|< (1/h+5)\lambda R$ and $y \in \Omega(2R)$,
$$
 \left| \dot{\eta}(\lambda;x) + \frac{y - \eta(\lambda ; x)}{\lambda} \right| \le 5R + \frac{|x-y|}{\lambda} + \left| \frac{x-\eta(\lambda;x)}{\lambda}\right| \le 16R,
$$
where we used \eqref{path_1} for the last inequality.
Hence, we further have
$$
|B| \leq 16 R \int_0^1\lambda^{-d-2}  \int_t^{t+(R/h)^2} \int_{\Omega(2R)} |\nabla u(l,y)|   |\phi(\cdot)|  \, dy  \, dl \, d\lambda.
$$

We now estimate the $L_{p_0}$ norms of $A_{1,k}$, $A_2$, and $B$ with respect to the spatial variables $x$ on $\Omega(R)$.
To estimate $A_{1,k}$, upon recalling that $\varphi$ has compact support in $B_{R/h}$, we note that, for each $\lambda \in [0,1]$ and $l \in [t,t+(R/h)^2]$,
$$
\int_{\Omega(2R)} \left| f_k(l,y) \right|   \varphi\left(\frac{y - \eta(\lambda ; x)}{\lambda}\right) \, dy = \int_{B_{\lambda R/h}\left(\eta(\lambda;x)\right)} \left| f_k(l,y) \right|   \varphi\left(\frac{y - \eta(\lambda ; x)}{\lambda}\right) \, dy
$$
$$
\le \int_{\mr^d} |f_k\left(l, y\right)|  \tilde{\varphi}\left(\frac{x-y}{\lambda}\right) \, dy,
$$
where
$$
\tilde{\varphi}(x) = I_{B_{\left(5 + h^{-1} \right)R}},
$$
since $\left| \varphi \right| \le 1$ and $y \in B_{\lambda h^{-1}R}\left(\eta\left({\lambda; x}\right) \right)$ implies $ y \in B_{\lambda \left(5 + h^{-1} \right)R}\left(x \right)$.
Then, set $a_k \in [1,\infty]$ so that $1+1/p_0 = 1/p_k + 1/a_k$, where $a_k \in [1,\infty]$ is guaranteed by the conditions on $p_0$ and $p_k$.
By Young's convolution inequality we have
$$
\left\| \int_{\Omega(2R)} \left| f_k(l,y) \right| \varphi\left(\frac{y-\eta(\lambda;x)}{\lambda}\right) \, dy \right\|_{L_{p_0}(\Omega(R))} 
\leq  \|f_k(l,\cdot)\|_{L_{p_k}(\Omega(2R))} \|\tilde{\varphi}\left(\cdot/\lambda \right) \|_{L_{a_k}(\mr^d)}
$$
$$
\leq N(d) R^{d+d/p_0-d/p_k}\lambda^{d+d/p_0-d/p_k} \|f_k(l,\cdot)\|_{L_{p_k}(\Omega(2R))}.
$$
Applying these inequalities to \eqref{eq0601_01} along with Minkowski's inequality, we get
\begin{equation}
							\label{eq0525_01}
\|A_{1,k}\|_{L_{p_0}(\Omega(R))} \leq N(d) R^{d+2+d/p_0-d/p_k} \int_0^1 \lambda^{d/p_0-d/p_k-1} F_k(\lambda,t) \, d\lambda
\end{equation}
for each $t \in \mr$, where
$$
F_k(\lambda, t) = \int_{t}^{t + (R/h)^2} \|f_k(l, \cdot)\|_{L_{p_k}(\Omega(2R))} \, \zeta\left( \frac{l-t}{\lambda^2}\right)\,dl.
$$

To estimate $A_2$ and $B$, we set $a \in [1,\infty]$ so that $1+1/p_0 = 1/p + 1/a$.
By proceeding similarly as in the estimate for $A_{1,k}$ above along with the fact that $|D \varphi| \leq N R^{-1}$, we obtain
\begin{equation}
	\label{eq0525_02}
\|A_2\|_{L_{p_0}(\Omega(R))} \le N R^{d + 1 + d/p_0 - d/p} \int_0^1 \lambda^{-2 + d/p_0 - d/p} \, G_i (\lambda, t)\,d\lambda
\end{equation}
and
\begin{equation}
	\label{eq0525_03}
\|B\|_{L_{p_0}(\Omega(R))} \le N R^{d + 1 + d/p_0 - d/p} \int_0^1 \lambda^{-2 + d/p_0 - d/p} \, U(\lambda, t)\,d\lambda,
\end{equation}
for each $t \in \mr$, where $N = N(d)$,
$$
G_i (\lambda, t) = \int_{t}^{t + (R/h)^2} \|g_i(l, \cdot)\|_{L_{p}(\Omega(2R))} \, \zeta\left( \frac{l-t}{\lambda^2}\right)\,dl,
$$
and
$$
U(\lambda, t) = \int_{t}^{t + (R/h)^2} \|\nabla u(l, \cdot)\|_{L_{p}(\Omega(2R))} \, \zeta\left( \frac{l-t}{\lambda^2}\right)\,dl.
$$

Now we are ready to derive the mixed norm estimates of $A_{1,k}$, $A_2$, and $B$.
To estimate $\|A_{1,k}\|_{L_{p_0,q_0}(\mr \times \Omega(R))}$, for each $k = 1, \ldots, m$, we consider the following two cases
$$
\textrm{(i)} \,\, 2 + \frac{d}{p_0} + \frac{2}{q_0} > \frac{d}{p_k} + \frac{2}{q_k}, \quad \textrm{(ii)} \,\, 2 + \frac{d}{p_0} + \frac{2}{q_0} = \frac{d}{p_k} + \frac{2}{q_k}.
$$
\begin{enumerate}[(i)]
\item 
Let $\varepsilon_k := 2 + d/p_0 + 2/q_0 - d/p_k - 2/q_k > 0$ and set $b_k \in [1,\infty]$ so that $1 + 1/q_0 = 1/q_k + 1/b_k$.
Then applying Young's convolution inequality to $F_k(\lambda, t)$ in \eqref{eq0525_01} with respect to $t$ variable, we get
\begin{equation}
							\label{eq0601_02}
\|F_k(\lambda,\cdot)\|_{L_{q_0}(\mr)} \leq \|f_k\|_{L_{p_k,q_k}\left(\mr \times \Omega(2R)\right)} \|\zeta_\lambda\|_{L_{b_k}(\mr)},
\end{equation}
where
$$
\|\zeta\|_{L_{b_k}(\mr)}^{b_k} = \int_{\mr} |\zeta(t/\lambda^2)|^{b_k} \, dt = \int_0^{\lambda^2 (R/h)^2} |\zeta(t/\lambda^2)|^{b_k} \, dt \leq \lambda^2 (R/h)^2.
$$
From this with Minkowski's inequality it follows that
$$
\| A_{1, k} \|_{L_{p_0, q_0}(\mr \times \Omega(R))} \le N(d)R^{d + 2 +  \varepsilon_k}\| f_k \|_{L_{p_k, q_k}(\mr \times \Omega(2R))}\int_0^1 \lambda^{\varepsilon_k -1}\,d\lambda 
$$
$$
\le NR^{d + 2 + \varepsilon_k} \| f_k \|_{L_{p_k, q_k}(\mr \times \Omega(2R))}.
$$

\item We first consider $1 < q_k < q_0 < \infty$. In this case, we have $1/q_k - 1/q_0 \in (0,1)$ and $d/p_0 - d/p_k < 0$.
Since $\operatorname{supp} \zeta \subset \left(0, (R/h)^2\right)$,
$$
 \int_0^1 \lambda^{-1 + d/p_0 - d/p_k}  F_k(\lambda, t)\, d\lambda
$$
$$
 \le \int_t^{t + (R/h)^2} \|f_k(l, \cdot)\|_{L_{p_k}(\Omega(2R))} \int_0^1  \lambda^{-1 + d/p_0 - d/p_k}  \zeta\left( \frac{l-t}{\lambda^2}\right)\,d\lambda \, dl
$$
$$
\le \int_{\mr} \|f_k(l, \cdot)\|_{L_{p_k}(\Omega(2R))} \int_{hR^{-1}|l-t|^{1/2}}^\infty  \lambda^{-1 + d/p_0 - d/p_k} \,d\lambda \, dl
$$
$$
\le NR^{d/p_k - d/p_0} \int_{\mr} |t - l|^{(d/p_0 - d/p_k)/2}  \|f_k(l, \cdot)\|_{L_{p_k}(\Omega(2R))}\,dl.
$$
Thus, we have
$$
\|A_{1, k} \|_{L_{p_0}(\Omega(R))} \leq NR^{d + 2} \int_{\mr}|t - s|^{-1 + 1/q_k - 1/q_0}  \|f_k(s, \cdot)\|_{L_{p_k}(\Omega(2R))} \,ds
$$
for each $t \in  \mr$.
Because $1/q_k - 1/q_0 \in (0,1)$ we see that $|t - s|^{-1 + 1/q_k - 1/q_0}$ is a 1-dimensional Riesz's potential.
Then by the Hardy-Littlewood-Sobolev theorem of fractional integration (see, for instance, \cite[p.119, Theorem 1]{MR0290095}), we arrive at
\begin{equation}
							\label{eq0601_03}
\| A_{1, k} \|_{L_{p_0, q_0}(\mr \times \Omega(R))} \le N R^{d + 2} \|f_k \|_{L_{p_k, q_k}(\mr \times \Omega(2R))}
\end{equation}
where $N = N(d, p_0, q_0, p_k, q_k)$.
For the case $(q_0, q_k) = (\infty,1)$, which requires $p_k = p_0$, we further assume that
$$
\operatorname{supp}\zeta \subset \left(\frac{1}{2}\left(\frac{R}{2h}\right)^2, \left(\frac{R}{h}\right)^2\right).
$$
Then
$$
\int_0^1 \lambda^{-1+d/p_0-d/p_k} F_k(\lambda, t) \, d\lambda
$$
$$
= 
\int_t^{t+(R/h)^2} \|f_k(l,\cdot)\|_{L_{p_k}(\Omega(2R))} \int_0^1 \lambda^{-1} \zeta\left(\frac{l-t}{\lambda^2}\right) \, d \lambda \, dl,
$$
where 
$$
\int_0^1 \lambda^{-1} \zeta\left(\frac{l-t}{\lambda^2}\right) \, d \lambda \leq \int_{hR^{-1}(l-t)^{1/2}}^{\sqrt{8}hR^{-1}(l-t)^{1/2}} \lambda^{-1} \, d \lambda = \ln \sqrt{8}.
$$
Thus, we have the same estimate as in \eqref{eq0601_03} with $q_k = 1$ and $q_0 = \infty$.
\

\end{enumerate}

To estimate the mixed norms of $A_2$ and $B$, we consider the following two cases:
$$
\textrm{(i)} \,\, 1 + \frac{d}{p_0} + \frac{2}{q_0} > \frac{d}{p} + \frac{2}{q}, \quad \textrm{(ii)} \,\, 1 + \frac{d}{p_0} + \frac{2}{q_0} = \frac{d}{p} + \frac{2}{q}.
$$
\begin{enumerate}[(i)]
\item 
Let $\varepsilon := 1 + d/p_0 + 2/q_0 - d/p - 2/q > 0 $ and set $b \in [1,\infty]$ so that $1 + 1/q_0 = 1/q + 1/b$.
Then by Young's convolution inequality applied to $G_i$ and $U$ in \eqref{eq0525_02} and \eqref{eq0525_03} with respect to $t$ variable, as in \eqref{eq0601_02} we obtain
$$
\|G_i(\lambda,\cdot)\|_{L_{q_0}(\mr)} \leq \lambda^{2/b} R^{2/b} \|g_k\|_{L_{p,q}\left(\mr \times \Omega(2R)\right)},
$$
$$
\|U(\lambda,\cdot)\|_{L_{q_0}(\mr)} \leq \lambda^{2/b} R^{2/b} \|\nabla u\|_{L_{p,q}\left(\mr \times \Omega(2R)\right)}.
$$
It then follows that
$$
\| A_2 \|_{L_{p_0, q_0}(\mr \times \Omega(R))} \le N(d)R^{d + 2 + \varepsilon}\| g \|_{L_{p, q}(\mr \times \Omega(2R))}\int_0^1 \lambda^{\varepsilon -1}\,d\lambda 
$$
$$
\le NR^{d + 2 + \varepsilon} \| g \|_{L_{p, q}(\mr \times \Omega(2R))},
$$
where $N = N(d, p, q, p_0, q_0)$.
Similarly, we have
$$
\| B \|_{L_{p_0, q_0}(\mr \times \Omega(R))} \le NR^{d + 2 + \varepsilon} \| Du \|_{L_{p, q}(\mr \times \Omega(2R))}.
$$
where $N = N(d, p, q, p_0, q_0)$.

\item
In this case, since $1<q<q_0<\infty$, we see that $-1+d/p_0-d/p < 0$.
By the same reasoning as in the estimate of $\|A_{1, k}\|_{L_{p_0, q_0}(\mr \times \Omega(R))}$ for the case (ii), we have
$$
\|A_2 \|_{L_{p_0}(\Omega(R))} + \|B \|_{L_{p_0}(\Omega(R))}
$$
$$
\leq NR^{d + 2} \int_{\mr}|t - s|^{-1 + \frac{1}{q} - \frac{1}{q_0}}  \left( \left\||g(s, \cdot)| + |\nabla u(s,\cdot)|\right\|_{L_{p}(\Omega(2R))} \right) \,ds.
$$
Since $1/q- 1/q_0 \in (0,1)$, $|t - s|^{-1 + 1/q - 1/q_0}$ is a 1-dimensional Riesz's potential.
Hence, by the Hardy-Littlewood-Sobolev theorem of fractional integration, we arrive at
$$
\| A_2 \|_{L_{p_0, q_0}(\mr \times \Omega(R))} + \| B \|_{L_{p_0, q_0}(\mr \times \Omega(R))} 
$$
$$
\le N R^{d + 2} \left(   \|g \|_{L_{p, q}(\mr \times \Omega(2R))} + \|Du \|_{L_{p, q}(\mr \times \Omega(2R))} \right)
$$
where $N = N(d, p, q, p_0, q_0)$.
\end{enumerate}

Finally, by combining the above mixed norm estimates for $A_{1,k}$, $A_2$, and $B$ with \eqref{mean} and \eqref{phi} along with the triangle inequality, we obtain
\begin{multline*}
\|u\|_{L_{p_0,q_0}\left(\mr \times \Omega(R)\right)} \leq N R^{d/p_0 + 2/q_0 - d/p - 2/q} \|u\|_{L_{p,q}(\mr \times \Omega(2R))}
\\
+ NR^{1 + d/p_0 + 2/q_0 - d/p - 2/q} \left( \|Du\|_{L_{p, q}\left(\mr \times \Omega(2R)\right)} +  \|g\|_{L_{p, q}\left(\mr \times \Omega(2R)\right)} \right)
\\
+N\sum_{k=1}^m  R^{2 + d/p_0 + 2/q_0 - d/p_k - 2/q_k} \|f_k\|_{L_{p_k, q_k}\left(\mr \times \Omega(2R)\right)},
\end{multline*}
where $N = N(d, p, q, p_0, q_0, p_k, q_k, m)$.
Indeed, the term $ \bar{u} $ is estimated  as follows.
$$
|\bar{u}| \leq N R^{-d-2} \int_{\mr} \zeta(s-t) \int_{B_{R/h}(z_0)} |u(s,z)| \, dz \, ds
$$
$$
\leq N R^{-d/p-2} \int_{\mr} \zeta(s-t) \|u(s,\cdot)\|_{L_p(\Omega(2R))} \, ds
$$
where $N = N(d)$.
Thus, by Young's convolution inequality with respect to $t$, we get
\begin{align*}
\|\bar{u}\|_{L_{p_0, q_0}(\mr \times \Omega(R))} &\leq   N R^{d/p_0-d/p-2} \|\zeta\|_{L_b(\mr)} \|u\|_{L_{p,q}(\mr \times \Omega(2R))}
\\
&\le N R^{d/p_0 - d/p + 2/q_0 - 2/q }\|u\|_{L_{p,q}(\mr \times \Omega(2R))},
\end{align*}
where $b \in [1,\infty]$ satisfies $1 + 1/q_0 = 1/q + 1/b$ and $\|\eta\|_{L_b(\mr)} \leq N R^{2/b}$.

The lemma is proved.
\end{proof}

\begin{lemma}[Interior estimates]
							\label{lem0805_2}
Let $p, q \in [1,\infty]$ and $\Omega \subset \mr^d$ be a domain.
Also let $p_0 \in [p,\infty]$, $q_0 \in (q,\infty]$, $p_k \in (0,\infty]$, $q_k \in (0,\infty]$, $k=1,\ldots,m$, satisfy the condition (ii) in Assumption \ref{dom}.
Then for $u \in W_{p,q}^{1, 0}(\Omega_T)$ and for any $\rho > 0$, if $u_t = D_ig_i + \sum_{k=1}^m f_k$ where $g = (g_1, \ldots, g_d) \in \left( L_{p, q}(\Omega_T) \right)^d$ and $f_k \in L_{p_k, q_k}(\Omega_T)$, $k = 1, \ldots, m$, we have
$$
\|u\|_{L_{p_0,q_0}((0, T) \times \Omega^{\rho})} \leq N \left( \|u\|_{W_{p,q}^{1, 0}(\Omega_T)} + \|g\|_{L_{p, q}(\Omega_T)} + N \sum_{k=1}^m \| f_k \|_{L_{p_k, q_k}(\Omega_T)}\right) , 
$$
where $N = N(d, p,q,p_0,q_0, p_k, q_k, m, T, \rho)$ and $\Omega^\rho = \{x \in \Omega: \operatorname{dist}(x,\partial\Omega) > \rho \}$.
\end{lemma}

\begin{proof}
The proof is similar to that of Lemma \ref{lem0805_1}. 
Fix $(t, x) \in \mr \times \Omega^{2\rho}$.
Then for $s \in \mr$ and $z \in B_{\rho}(x)$, define a path $\tau$ from $(t, x)$ to $(s, z)$ by 
$$
\tau(\lambda) = \left((1-\lambda^2)t + \lambda^2s, \eta(\lambda)\right) = \left( (1-\lambda^2)t + \lambda^2s, (1-\lambda)x + \lambda z \right)
$$
on $\lambda \in [0, 1]$. Then one may proceed similarly as in the proof of Lemma \ref{lem0805_1}.
Note that $\tau(\lambda) \in \mr \times \Omega^{\rho}$, $B_{\lambda\rho}(\eta(\lambda)) \subset\subset \Omega$ for $\lambda \in [0, 1]$, and $|\eta'| =  |x -z| \le \rho$.
\end{proof}

\begin{proof}[Proof of Theorem \ref{thm602}]
Since the $(R_0, \gamma)$-Reifenberg flatness is stronger than the $(R_0, 1/48)$-Reifenberg flatness for $\gamma \le 1/48$, we fix $\gamma = 1/48$. 

Case $q_0=q$ : Note that $1 + d/p_0 \ge d/p$ and we exclude the case $(p, p_0) = (d, \infty)$. Then since $u(t, \cdot) \in W_p^1(\Omega)$ for almost every $t \in (0, T)$ and $\Omega$ is an extension domain (see \cite{MR631089, MR1446617}),
$$
\| u(t, \cdot) \|_{L_{p_0}(\Omega)} \le N(p, p_0, d, R_0)\left(\| u (t, \cdot)\|_{L_p(\Omega)} + \| Du(t, \cdot) \|_{L_p(\Omega)}\right)
$$
for almost every $t \in (0, T)$. Then by integrating $q(=q_0)$-th power of both sides over $(0, T)$, we obtain \eqref{imb2} immediately.  If $p_0 = \infty$ and $p=d=1$, we directly obtain the same estimate by the fundamental theorem of calculus.

Case $q_0 > q$ : We derive \eqref{imb2} by using the partition of unity argument with respect to the $x$ variables, Lemma \ref{lem0805_1} with $R = R_0/4$, and Lemma \ref{lem0805_2} with a sufficiently small $\rho$ (for example, $\rho = R_0/384$). Also recall that $\Omega$ is bounded.
\end{proof}

One can also obtain the following embedding result when $\Omega = \mr^d$ or $\Omega = \mr^d_+$ by following similar steps as in the proof of Lemma \ref{lem0805_1}.
We omit the proof and leave the details to the interested reader.
Note that the local embedding results as in Lemma \ref{lem0805_1} may not be used via the partition of unity argument to derive embeddings for unbounded domains because we are dealing with mixed-norms.

\begin{theorem}[Embedding in $(0, T) \times \mr^d$ and $(0, T) \times \mr_+^d$]
	\label{thm702}
Let $p, q \in [1, \infty]$ and $\Omega = \mr^d$ or $\Omega=\mr_+^d$.
Also let $p_0 \in [p,\infty]$, $q_0 \in [q,\infty]$, $p_k \in (0,\infty]$, $q_k \in (0,\infty]$, $k=1,\ldots,m$, satisfy either the condition (i) or (ii) in Assumption \ref{dom}.
Then for $u \in W_{p, q}^{1, 0}(\Omega_T)$, if $u_t = D_ig_i + \sum_{k=1}^m f_k$ where $g = (g_1, \ldots, g_d) \in \left( L_{p, q}(\Omega_T) \right)^d$ and $f_k \in L_{p_k, q_k}(\Omega_T)$, we have
\begin{equation*}
\|u\|_{L_{p_0,q_0}(\Omega_T)} \le N \left( \|u\|_{W_{p,q}^{1, 0}(\Omega_T)} + \| g \|_{L_{p, q}(\Omega_T)} + \sum_{k=1}^m \| f_k \|_{L_{p_k, q_k}(\Omega_T)} \right)
\end{equation*}
where $N = N(d, p, q, p_0, q_0, p_k, q_k, m, T)$.
\end{theorem}

\begin{remark}
If the spatial domain is $\Omega = \mr^d$ or $\Omega = \mr^d_+$, the result in Theorem \ref{thm702} holds even when $q = 1$ in \eqref{eq0602_02}.
In this case we need
$$ 
\frac{d}{p_k} + \frac{2}{q_k} \le 1 + \frac{d}{p} + \frac{2}{q},
\quad \text{and} \quad
q_k = 1 \quad \text{if} \quad \frac{d}{p_k} + \frac{2}{q_k} = 1 + \frac{d}{p} + \frac{2}{q},
$$
from which the inequality \eqref{eq0603_02} (\eqref{eq0603_03} with $1=q_k < q_0 < \infty$) holds because
$d/p + 2/q \leq 1 + d/p_0 + 2/q_0$.
When $\Omega$ is a Reifenberg flat domain, the same case can be considered for $u$ having the zero lateral boundary condition.
However, even with such a boundary condition, one cannot use a zero extension of $u$ to the whole $\mr^d$ and the embedding for $\mr^d$ because, as noted in Remark \ref{rem0430_1}, the extension may not belong to the same class of functions.
\end{remark}

\begin{remark}
As shown above, if $\Omega$ is a bounded $(R_0, 1/48)$-Reifenberg flat domain with $\gamma \le 1/48$,
Assumption \ref{dom} ($\Omega$) holds with
$$
C = C(R_0, d, p, q, p_0, q_0, p_k, q_k, m, \operatorname{diam}(\Omega), T).
$$
Hence, Theorem \ref{theo401} holds with 
the constant 
$$
N = N(d, \delta, p,q,p_1,q_1, \ell_2,r_2,\ell_3,r_3, a^i, b^i, c, K, R_0, \operatorname{diam}(\Omega), T).
$$
If $\Omega = \mr^d$ or $\mr^d_+$, Assumption \ref{dom} ($\Omega$) holds with 
$$
C = C(d, p, q, p_0, q_0, p_k, q_k, m, T)
$$
and Theorem \ref{theo401} also holds with
$$
N = N(d, \delta, p, q, p_1, q_1, \ell_2,r_2,\ell_3,r_3, a^i, b^i, c, K, T).
$$
Also see Remark \ref{john2} for the conormal derivative problem.
\end{remark}

\begin{remark}[Elliptic problems]
In this paper Lemma \ref{190405@lem1} is the key lemma, where we use the duality argument in the proof.
Using the same duality argument and proceeding as in this paper, one can obtain analogous results for the elliptic case.
See \cite{MR3623550, MR3900848, MR3328143} for related results based on different approaches including a functional analytic approach.
\end{remark}

\section*{Acknowledgments}
The authors would like to thank Hongjie Dong for helpful suggestions and discussion including the full ranges of $(p_k,q_k)$ in Theorem \ref{thm602} and an example supporting the assertion in Remark \ref{rem0430_1}.
The authors also thank Jongkeun Choi for discussion on an early version of the paper.
\bibliographystyle{plain}

\def\cprime{$'$}

\end{document}